\documentclass[11pt,english]{smfart}

\usepackage{a4wide}

\usepackage[english]{babel}
\usepackage{amsmath}
\usepackage{amsfonts}
\usepackage{amssymb}
\usepackage{amsthm}
\usepackage{a4wide}
\usepackage{mathrsfs}
\usepackage[T1]{fontenc}
\usepackage{tikz}
\usepackage{shuffle}
\usepackage[all]{xy}
\usepackage{xcolor}
\usepackage[latin1]{inputenc}

\usepackage{nomencl}
\usepackage[intoc]{nomentbl}
\makenomenclature

\def\di{\displaystyle}
\newcommand{\R}{\mathbb{R}}
\newcommand\del[1]{}
\newcommand\think[1]{}
\newcommand\new[1]{}
\newcommand\zus[1]{}

\newcommand\comd[1]{} 
\newcommand\Redd[1]{} 

\def\bdm{\begin{displaymath}}
\def\edm{\end{displaymath}}
\def\bea{\begin{eqnarray}}
\def\eea{\end{eqnarray}}

\newtheorem{theorem}{Theorem}[section]
\newtheorem{lemma}[theorem]{Lemma}
\newtheorem{definition}[theorem]{Definition}
\newtheorem{proposition}[theorem]{Proposition}
\newtheorem{corollary}[theorem]{Corollary}
\newtheorem{example}{Example}
\newtheorem{remark}{Remark}

\begin{document}

\title[A stochastic Landau-Lifshitz equation]{Selection of a stochastic Landau-Lifshitz equation and the stochastic persistence problem}

\author{Jacky Cresson, Yasmina Kheloufi, Khadra Nachi, Frederic Pierret}
\address{Laboratoire de Math\'ematiques Appliqu\'ees de Pau, 
 Universit\'e de Pau et des Pays de l'Adour, 
Avenue de l'Universit\'e, BP 1155,
64013 Pau Cedex, France
\&
Syst\`emes de R\'ef\'erence Temps-Espace - UMR CNRS 8630, Observatoire de Paris,
61 avenue de l'Observatoire,
75014 Paris, France
}
\email{jacky.cresson@univ-pau.fr}
\address{Universit\'e d'Es-S\'enia, Oran -- D\'epartement de math\'ematiques, Alg\'erie}
\email{kheloufiyasmina@ymail.com}

\begin{abstract}
In this article, we study the persistence of properties of a given classical deterministic differential equation under a stochastic perturbation of two distinct forms: external and internal. The first case corresponds to add a noise term to a given equation using the framework of It\^o or Stratonovich stochastic differential equations. The second case corresponds to consider a parameters dependent differential equations and to add a stochastic dynamics on the parameters using the framework of random ordinary differential equations. Our main concerns for the preservation of properties is stability/instability of equilibrium points and symplectic/Poisson Hamiltonian structures. We formulate persistence theorem in these two cases and prove that the cases of external and internal stochastic perturbations are drastically different. We then apply our results to develop a stochastic version of the Landau-Lifshitz equation. We discuss in particular previous results obtain by Etore and al. in \cite{Etore} and we finally propose a new family of stochastic Landau-Lifshitz equations.
\end{abstract}

\maketitle

\tableofcontents

\section{Introduction}

\subsection{Stochastic extension of a deterministic model}

In an increasing number of domains, scientists need to take into account {\bf stochastic effects} in a problem which were first modeled in a {\bf deterministic} way. The reason is in general due to new experimental data which show a random behavior. This is the case for example in the study of the mechanisms underlying the Hodgkin-Huxley model describing the bioelectrical dynamics of neurons or the random behavior of the flattening of the earth. The previous situation covers different modeling problems. Essentially, we can distinguish three types:
\begin{itemize}
\item A parameter entering in the deterministic model which was assumed to be a constant has a random behavior.

\item A new mechanism which is random in nature enters in the description of the model.

\item The initial model is under the influence of external random fluctuations.
\end{itemize}

The third case is the classical way to take into account the {\bf environment} of the phenomenon which is observed. This situation corresponds to what we call an {\bf external stochastic perturbation}. \\ 

The second and the first one are related and are called {\bf internal stochastic perturbations} in the following. Indeed, the new mechanism can be for example the stochastic dynamical behavior of a parameter of the deterministic system. This is for example the case with the Hodgkin-Huxley model where the dynamic of the ion channels is known to be stochastic.\\

Formally, one considers a parameters dependent ordinary differential equation of the form 
\begin{equation}
\di\frac{dx}{dt}=f(t,x;p),\ x\in \R^n ,\ p\in \R^m .
\end{equation}
An external stochastic perturbation means that we consider the following situation
\begin{equation}
\di\frac{dx}{dt}=f(t,x;p) + "{\rm noise}",
\end{equation}
and for an internal stochastic perturbation, we consider the case
\begin{equation}
\left .
\begin{array}{lll}
\di\frac{dx}{dt} & = & f(t,x;p_t ),\ x\in \R^n ,\ p\in \R^m ,\\
p_t & = & p_0 +"{\rm noise}" .
\end{array}
\right .
\end{equation}
We can sometimes relate the two approaches but they are generically different. In particular, one can use different formalism to give a sense to these two situations.\\

As we will see, some fundamental questions related to modeling and in particular the {\bf persistence} of {\bf properties} of the initial deterministic system under the {\bf stochastisation process} depends drastically on the nature of the stochasticity, internal or external.  

\subsection{Classical approaches}
\label{classical}

There exists many ways to {\bf construct} a {\bf stochastic model}. A classical one is known as the {\bf Allen's derivation method} and is discussed for example in \cite{AlE,Al}. It is based on the assumption that the observed dynamics can be decomposed as the sum of elementary stochastic process which can be completely described. This is certainly the most {\bf natural} way and close to the standard modeling procedure of all the methods. We refer to \cite{CrSo} for a discussion of this method.\\

In this article, we will focus on a second one which is frequently used in the literature. The mathematical framework is given by the theory of {\bf stochastic differential equations} in It\^o or Stratonovich sense. Precisely, we consider a deterministic ordinary differential equation of the form
\begin{equation}
\label{ode}
\di\frac{dx}{dt}=f(x),\ \ x\in \mathbb{R}^n,\ n\in \mathbb{N}^* .
\end{equation}
A stochastic behavior is taken into account by adding a "noise" term to the classical deterministic equation as follows:
\begin{equation}
\di\frac{dx}{dt}=f(x) +"\mbox{\rm noise}",
\end{equation}
and to replace the "noise" term by a stochastic one as
\begin{equation}
\label{sde}
dX_t = f(X_t ) dt + \sigma (X_t ,t) dW_t ,
\end{equation}
where $W_t$ is standard Wiener process. This procedure is for example well discussed in (\cite{oksendal2003stochastic}). \\

Another way to introduce stochasticity can be obtained using the theory of {\bf random ordinary differential equations (RODEs)} as defined for example in \cite{kloeden2}. The specific procedure to introduce a stochastic component called the {\bf RODEs stochastisation procedure} is discussed in Section \ref{rodeapproach}.

\subsection{Selection of models and the stochastic persistence problem}

Of course, the main problem is in this case to {\bf find} the {\bf form of the stochastic perturbation}. We do not discuss this problem which is very complicated. We restrict our attention to the {\bf selection problem} which is concerned with the characterization of the set of {\bf admissible} stochastic models for a given phenomenon. By admissible we mean that the stochastic model satisfies some known constraints like positivity of some variables, conservation law, etc.  This {\bf selection} of a good candidate for a stochastic model of a phenomenon can be done in many ways. However, in our particular setting, dealing with the {\bf stochastic extension} of a known deterministic model, this selection is related to {\bf preserving} some specific {\bf constraints} of the phenomenon. For example, part of the Hodgkin-Huxley model describes the dynamical behavior of concentrations which are typically variables which belong to the interval $[0,1]$. This property is independent of the particular {\bf dynamics} of the variables but is related to their {\bf intrinsic} nature. The same is true for the total energy of a mechanical system. This quantity must be preserved independently of the dynamics. We formulate the {\bf stochastic persistence problem} following our approach given in \cite{CrSz} in a different setting: \\

\noindent {\bf Stochastic persistence problem} : {\it Assume that a classical ODE of the form (\ref{ode}) satisfies a set of properties $\mathcal{P}$. Under which conditions a stochastic model obtained using a specific stochastisation procedure satisfies also properties $\mathcal{P}$ ?}\\

In the case of the It\^o or Stratonovich stochastisation procedure (see Section \ref{classical}), the previous problem will lead to a {\bf characterization} of the set of $\sigma$ preserving the considered properties $\mathcal{P}$. In this article, we focus on the following properties :
\begin{itemize}
\item Invariance of a given submanifold of $\mathbb{R}^n$,
\item Number of equilibrium points,
\item Stability properties of the equilibrium points .
\item Preserving Symplectic or Poisson Hamiltonian structures
\end{itemize}
All these problem have been studied so that a complete answer can be given. The invariance problem was specifically discussed in \cite{cy}. In this article, we discuss the {\bf stability/instability} property of equilibrium points in Section \ref{stabilitypart} following previous work of Khasaminskii \cite{Khasminskii} in the It\^o/Stratonovich case and Han, Kloeden in \cite{kloeden2} for the RODEs case, and {\bf Symplectic/Poisson Hamiltonian structures} in Section \ref{symppoisson} following previous works of J-M. Bismut in \cite{bismut} and Lazaro-Cami, Ortega in \cite{cami} in the Stratonovich setting. 

\subsection{Constructing a stochastic Landau-Lifshitz equation}

In order to illustrate the previous problems on a concrete example, we study in Part \ref{landau} the selection of a {\bf stochastic Landau-Lifshitz} equation. The Landau-Lifshitz (LL) equation \cite{Federer} describes the evolution of magnetization in continuum ferromagnets and plays a fundamental role in the understanding of non-equilibrium magnetism. Following \cite{Etore}, the Landau-Lifshitz equation is given by
 
\begin{equation}
   \label{LLg}
   \frac{d\mu}{dt}=-\mu\wedge h -\alpha\mu \wedge(\mu \wedge h),
   \tag{LLg}
\end{equation} 
where $\mu(t)$ is the single magnetic moment, $\wedge$ is the vector cross product in $ \mathbb{R}^3$, $h$ is the effective field, and $\alpha >0$ the damping effects.\\
   
The LL equation has some important properties \cite{Cimrak}:
\begin{itemize}
\item The main one is that the norm of the magnetization is a constant of motion, i.e. that up to normalization the sphere $S^2$ is an invariant submanifold for the system. 

\item Secondly, the free energy of the system is a nonincreasing function of time. This property is also fundamental, because it guarantees that the system tends toward stable equilibrium points, which are minima of the free energy. 

\item Lastly, if there is no damping, i.e., $\alpha=0,$ the energy is conserved, which is called the Hamiltonian structure \cite{Lakshmanan} and \cite{Bertotti}.
\end{itemize}

However, as reminded by Etore and al. in (\cite{Etore}, Section.1), "in order to understand the behavior of ferromagnetic materials at ambient temperature or in electronic devices where the Joule effect induces high heat fluxes" one need to study {\bf thermal effects} in ferromagnetic materials. This effect is usually modeled by the introduction of a noise at microscopic scale on the magnetic moment direction and the mesoscopic scale by a transition of behavior. As a consequence, as stated in \cite{Etore}, "{\bf it is essential to understand the impact of introducing stochastic perturbations in deterministic models of ferromagnetic materials} such as the micro-magnetism" (\cite{brown1,brown2}).\\
 
During the last decade, an increasing number of articles have been devoted to the {\bf construction of a stochastic version} of the LL equation (see \cite{mercer,zheng,raikher,atkinson,stepanov,scholz,martinez,smith}).\\

Up to our knowledge, all the models constructed in the previous articles are considering {\bf external stochastic perturbations}. Essentially, we have two kind of stochastic models which are related to the stochastic calculus framework used: It\^{o} or Stratonovich stochastic calculus. Both formalism has its advantages and problems.

\begin{itemize}
\item The Stratonovich formalism is suitable for problems involving geometric questions. Indeed, as the Stratonovich calculus behaves like the classical differential calculus, one can recover most known results in differential geometry. This in particular the case for invariance of manifolds. However, other issues related to equilibrium points and stability problems lead to many difficulties. This is in particular the case for all the Stratonovich version of the LL equation as discussed for example by Etore and al in (\cite{Etore}, Section 2.2): the Stratonovich stochastic LL equation satisfies trivially the invariance condition. However, as proved in (\cite{Etore}, Proposition 4), we have not the expected equilibrium point $b$ and moreover, the dynamics is positive recurrent on the whole sphere except a small northern cap. Moreover, the interpretation is more difficult (see \cite{Etore}, section 2.2).

\item The It\^{o} formalism is up to our knowledge, only used by Etore and al in \cite{Etore} in order to solve the previous problems. They use the $S^2$ version of an invariantization method discussed in \cite{cy}. However, as already discussed, this method is unsatisfactory (see \cite{cy}, Section 4). 
\end{itemize}

Using the RODE's approach discussed in Section \ref{rodeapproach}, we propose an {\bf alternative modeling}. Indeed, even if the previous models use external stochastic perturbations, they are intended to model the stochastic fluctuation of the external field which enters as a parameter in the deterministic model. As a consequence, we can adopt our strategy to model the stochastic fluctuation of the externalfield as an {\bf internal stochastic fluctuation}. This is precisely done in Section \ref{rode}, where we use the theory of RODE for random ordinary differential equation to formulate a new class of stochastic LL equation satisfying the invariance of $S^2$ and possessing good properties for the stability/instability of the equilibrium points. 

\subsection{Organization of the paper}

Section \ref{remindstoc} contains a short reminder of classical  results on It\^o and Stratonovich's stochastic differential equations.

We then define and study two stochastisation procedures: In Section \ref{stochastisationparameter} the external stochastisation of linearly parameter dependent systems and in Section \ref{rodeapproach} the RODEs stochastisation procedure. These methods are used to discuss stochastic version of the Larmor equation which is a non dissipative version of the Landau-Lifshitz equation. \\

In Section \ref{stabilitypart} and Section \ref{symppoisson}, we study the persistence problem from the point of view of stability and the Hamiltonian structure of an equation. We use the It\^o, Stratonovich and RODEs formalism. Several examples are also given for each notion. This part is self-contained.\\

Section \ref{landau} discuss the construction of a stochastic Landau-Lifshitz equation following some selection rules which are assumed to be fundamental for all possible stochastic generalization of the equation. We give a self-contained introduction to the properties of the Landau-Lifshitz equation. We review classical approach to the  stochastisation process and we propose a new model based on the RODEs approach discuss in Section \ref{rode}.
 
\section{Reminder about stochastic differential equations}
\label{remindstoc}

In this article, we consider a parameterized differential equation of the form 

\begin{equation}
\label{DE}
dX_t = f (t,X_t,b)dt,\quad x\in \R^{n}
\tag{DE}
\end{equation}

where $b\in \R^{k}$ is a set of parameters, $f: \R^{n}\times \R^{k}\longrightarrow \R^{n}$ is a Lipschitz continuous function with respect to $x$ for all $b$.
We remind basic properties and definition of stochastic differential equations in the sense of It\^o and of Stratonovich. We refer to the book \cite{oksendal2003stochastic} for more details.

\subsection{It\^o stochastic differential equation}

A {\it stochastic differential equation} is formally written (see \cite{oksendal2003stochastic},Chap.V) in differential form as  

\begin{equation}
\label{IE}
dX_t = f (t,X_t)dt+\sigma(t,X_t)dB_t ,
\tag{IE}
\end{equation}

which corresponds to the stochastic integral equation
\begin{equation}
X_t=X_0+\int_0^t f (s,X_s)\,ds+\int_0^t \sigma (s,X_s)\,dB_s ,
\end{equation}
where the second integral is an It\^o integral (see \cite{oksendal2003stochastic},Chap.III) and $B_t$ is the classical Brownian motion (see \cite{oksendal2003stochastic},Chap.II,p.7-8).\\

An important tool to study solutions to stochastic differential equations is the {\it multi-dimensional It\^o formula} (see \cite{oksendal2003stochastic},Chap.III,Theorem 4.6) which is stated as follows : \\

We denote a vector of It\^o processes by $\mathbf{X}_t^\mathsf{T} = (X_{t,1}, X_{t,2}, \ldots, X_{t,n})$ and we put $\mathbf{B}_t^\mathsf{T} = (B_{t,1}, B_{t,2}, \ldots, B_{t,l})$to be a $l$-dimensional Brownian motion (see \cite{karatzas},Definition 5.1,p.72),  $d\mathbf{B}_t^\mathsf{T} = (dB_{t,1}, dB_{t,2}, \ldots, dB_{t,l})$. We consider the multi-dimensional stochastic differential equation defined by (\ref{IE}). Let $F$ be a $\mathcal{C}^2(\mathbb{R}_+ \times \mathbb{R},\mathbb{R})$-function and $X_t$ a solution of the stochastic differential equation (\ref{IE}). We have 
\begin{equation}
dF(t,\mathbf{X}_t) = \frac{\partial F}{\partial t} dt + (\nabla_\mathbf{X}^{\mathsf T} F) d\mathbf{X}_t + \frac{1}{2} (d\mathbf{X}_t^\mathsf{T}) (\nabla_\mathbf{X}^2 F) d\mathbf{X}_t,
\end{equation}
where $\nabla_\mathbf{X} F = \partial F/\partial \mathbf{X}$ is the gradient of $F$ w.r.t. $X$, $\nabla_\mathbf{X}^2 F = \nabla_\mathbf{X}\nabla_\mathbf{X}^\mathsf{T} F$ is the Hessian matrix of $F$ w.r.t. $\mathbf{X}$, $\delta$ is the Kronecker symbol and the following rules of computation are used : $dt dt = 0$, $dt dB_{t,i}  = 0$, $dB_{t,i} dB_{t,j} = \delta_{ij} dt$.

\subsection{Stratonovich stochastic differential equation}

A Stratonovich stochastic differential equation is formally denoted in differential form by
\begin{equation}
\label{SE}
dX_t = f (t,X_t ) dt +\sigma (t,X_t ) \circ dB_t ,
\tag{SE}
\end{equation}
which corresponds to the stochastic integral equation 
\begin{equation}
X_t =x+\di\int_0^t f (s, X_s ) ds + \di\int_0^t \sigma (s, X_s ) \circ dB_t ,
\end{equation}
where the second integral is a Stratonovich integral (see \cite{oksendal2003stochastic},p.24,2)).\\

The advantage of the Stratonovich integral is that it induces classical chain rule formulas under a change of variables:
\begin{equation}
dF(t,\mathbf{X}_t) = \frac{\partial F}{\partial t} dt + (\nabla_\mathbf{X}^{\mathsf T} F) d\mathbf{X}_t .
\end{equation}

\subsection{Conversion formula}

A solutions of the Stratonovich differential equation (\ref{SE}) corresponds to the solutions of a modified It\^o equation (see \cite{oksendal2003stochastic},p.36) :
\begin{equation}
\label{ito}
dX_t = f_{\rm cor} (t,X_t ) dt +\sigma (t,X_t )  dB_t ,
\end{equation}
where
\begin{equation}
f_{\rm cor} (t,x ) =\left [ f (t,x ) +\di\frac{1}{2} \sigma' (t,x ) \sigma (t,x) \right ] .
\end{equation}
The correction term $\di\frac{1}{2} \sigma' (t,X_t ) \sigma (t,X_t)$ is also called the {\it Wong-Zakai} correction term.
In the multidimensional case, i.e., $f :\R^{n+1} \rightarrow \R^n$, $f (t,x)=(f_1 (t,x) ,\dots ,f_n (t,x))$ and $\sigma :\R^{n+1} \rightarrow \R^{n\times l}$, 
$\sigma (t,x)=( \sigma_{i,j} (t,x))_{1\leq i\leq n,\ 1\leq j\leq l}$ the analogue of this formula is given by (see \cite{oksendal2003stochastic},p.85) :
\begin{equation}
\label{strato-ito}
f_{\rm cor, i} (t,x)=f_i (t,x)+\di\frac{1}{2} \di\sum_{j=1}^n \di\sum_{k=1}^l \di\frac{\partial \sigma_{i,j}}{\partial x_k} \sigma_{k,j},\ \ 1\leq i \leq n .
\end{equation}
In all this paper, let assume that the functions $f,g$ satisfy the assumptions of the existence and uniqueness theorem for stochastic differential equations \cite[theorem 5.2.1, p 66] {oksendal2003stochastic}, where they are continuous in both variables, satisfy a global Lipschitz condition with respect to the variable $x \in \mathbb{R}^{n}$ uniformly in $t,$  i.e., there exists a constant $L>0$ such that
$$ \vert f(t,x_1) -f(t,x_2)   \vert + \vert \sigma(t,x_1) -\sigma(t,x_2) \vert\leq L \vert x_1-x_2\vert;$$
for all $x_1, x_2 \in \mathbb{R}^{n}$ and for all $t\geq 0,$ and bounded in the sense that 
$$ \vert f(t,x)  \vert + \vert \sigma(t,x)  \vert\leq C (1+ \vert x \vert); \quad x\in \mathbb{R}^{n}, t\geq 0$$
for some constant $C.$ In addition, suppose that the initial condition $x_0$ is an adapted process to the filtration $(\mathcal{F}_t)$ generated by $W_t(.), t\geq 0$ with $\mathbb{E} [x_0^{2}]< \infty.$

\section{External stochastisation of linearly parameter dependent systems}
\label{stochastisationparameter}

The previous framework can be used to produce a stochastic version of a deterministic system which depends {\bf linearly} on a parameter which is assumed to behave stochastically.

\subsection{The stochastisation procedure}
\label{sectionexternal}

Let us consider an ordinary differential equation of the form
\begin{equation}
\label{deterinitial}
\di\frac{dx}{dt} =f(x;p) ,\ x\in \R^n,\ p\in \R^m ,
\end{equation}
where $f$ is sufficiently smooth and which is {\bf linear} in the parameter $p$. As a consequence, we can write $f$ as 
\begin{equation}
f(x;p)=A(x).p , 
\end{equation}
with $A(x)\in \mathcal{M} (n,m)$.\\

We assume that $p$ behaves stochastically and we apply the previous procedure to produce a stochastic version of the equation. We then denote 
\begin{equation}
p=p_0 +N (x,t) ,
\end{equation}
where $N$ corresponds to the noise. As $f$ is linear, we have 
\begin{equation}
f(x;p)=f(x;p_0) +f(x;N(x,t)) ,
\end{equation}
and we are lead formally to an equation of the form
\begin{equation}
\label{noiselinear}
\di\frac{dx}{dt} =f(x;p_0 )+f(x;N(x,t)).
\end{equation}
As remind in the introduction, the noise $N$ must be seen as the "derivative" of a white noise. Using the framework of stochastic differential equations, we obtain the following {\bf stochastic version} of equation  (\ref{noiselinear}):
\begin{equation}
\label{noiselinearstoc}
dX_t =f(X_t ;p_0 )dt+f(x;\sigma (x,t) dW_t),
\end{equation}
where $\sigma (x,t) \in \mathcal{M}(m,k)$ and $W_t$ is a $k$ dimensional Brownian motion.\\

The previous procedure will be called {\bf stochastisation} of equation (\ref{noiselinear}) under external stochastic perturbations on the parameters.

\subsection{Properties of the stochastized equation}

Some important properties are preserved by the stochastisation procedure.\\

Let $p=p_0$ be fixed. An {\bf equilibrium point} of the deterministic equation (\ref{deterinitial}) denoted $x_{\star} (p_0 )$ is a solution of the equation
\begin{equation}
f(x_{\star} ;p_0 )=0 .
\end{equation}
We denote by $\mathcal{E}_{p_0}$ the set of equilibrium points of equation (\ref{deterinitial}) when $p=p_0$.\\

Due to the linearity of $f$ with respect to $p$, we have:

\begin{lemma}[Persistence of equilibrium points] 
\label{persistenceexternalequilibrium}
The set of equilibrium point of the stochastized equation (\ref{noiselinearstoc}) coincides with those of the deterministic equation independently of the It\^o or Stratonovich interpretation of the SDE if and only if the diffusion term $\sigma (x,t) dW_t$ is such that $\sigma (t,x_{\star} ) dW_t$ belongs to the Kernel of the linear maps $p\rightarrow f(x_{\star} ,p)$ for $x_{\star} \in \mathcal{E}_{p_0}$.
\end{lemma}

\begin{proof}
We begin with the It\^o case. An equilibrium point $x_{\star}$ must satisfies the two equations $f(x_{\star} ,p_0 )=0$ and $f(x_{\star} ;\sigma (x_{\star} ,t) dW_t) =0$. The first one is equivalent to $x_{\star}$ is in $\mathcal{E}_{p_0}$. Due to the linearity of $f$, the second equation is satisfied if and only if $\sigma (x_{\star} ,t )dW_t$ belongs to the Kernel of the linear function $p\rightarrow f(x_{\star} ;p)$ for each $x_{\star}$ in $\mathcal{E}_{p_0}$.\\

In the Stratonovich case we return to the It\^o case using the conversion formula. The correcting term involves the diffusion coefficient $A(x) \sigma (t,x)$ evaluated at the point $x_{\star}$ and its partial derivatives. As $A(x_{\star} )\sigma (t,x_{\star} )$ must be $0$ for all $t$, we deduce that the correcting term is also zero in the drift part. As a consequence, the drift condition is equivalent to $f(x_{\star} ,p_0 )=0$ as in the previous case. This concludes the proof.
\end{proof} 

The previous result implies that in the setting of external stochastic perturbation of linearly parameter dependent systems, the choice of the It\^o or Stratonovich framework can not be decided only looking to the set of equilibrium points. We will see that the asymptotic behavior of the solutions can nevertheless be used to select a stochastic model.

\subsection{A stochastic Larmor equation}
\label{larmor}

The {\bf Larmor equation} defined for all $b\in \R^3$ by 
\begin{equation}
\label{eqlarmor}
\di\frac{d\mu }{dt} =\mu \wedge b ,\ \ \mu\in \R^3 .
\end{equation}
For each $b\in \R^3$, equilibrium points are solution of the equation $\mu \wedge b=0$ which gives $\mu \in \langle b \rangle$ Moreover, we have for all $\mu_0 \in \R^3$, the solution $\mu_t$ beginning with $\mu_0$ is such that $\parallel \mu_t \parallel =\parallel \mu_0 \parallel$. As a consequence, for each $r>0$, the sphere centered in $O=(0,0,0)\in \R^3$ of radius $r$ is invariant under the flow of the Larmor equation. As a consequence, restricting our attention to solutions beginning with $\mu_0 \in S^2$, we have for all $b\in \R^3$:
\begin{itemize}
\item The sphere $S^2$ is invariant under the flow of the Larmor equation.
\item The flow of the Larmor equation restricted to $S^2$ possesses two equilibrium point given by $\pm b/\parallel b \parallel$.
\item The function $V(\mu )=\mu .b$ is a first integral of the Larmor equation.
\end{itemize}
The two equilibrium points of the Larmor equation are {\bf center} equilibrium points, i.e. equilibrium points surrounded by family of concentric periodic orbits. This situation is a priori highly sensitive to perturbation. The two equilibrium points are {\bf stable} equilibrium points (see Section \ref{remindstabilitydeter} for a reminder about stability of equilibrium points). The motion of the magnetic moment $\mu$ describing circle around the axes defined by $b$ is called the {\bf Larmor precession}.\\

Using the previous result on persistence, one can study the effect of an external perturbation on the Larmor equation. An external perturbation of the parameter $b$ in the Stratonovich setting of the Larmor equation leads to an equation of the form (see Section \ref{sectionexternal}):
\begin{equation}
\label{stoclarmor}
d\mu = \left [ \mu \wedge b \right ]  dt+ \epsilon \mu \wedge \left [ \sigma (t,\mu ) d_{\circ} W_t \right ],
\end{equation}
where $W_t$ is a three dimensional Brownian motion and $\sigma \in \mathcal{M} (3,3)$ is a non zero matrix. Such an equation will be called a {\bf stochastic Larmor equation}.

\begin{lemma}
\label{larmorinvariance}
The flow of any stochastic Larmor equation preserves the sphere $S^2$. 
\end{lemma}

\begin{proof}
A general result asserts that a given submanifold $M$ defined by a function $F$, i.e. $M=\{ x\in \R^n \mid F(x)=0 \}$  which is invariant under the deterministic flow associated to \eqref{DE}, i.e., 
$$\nabla F(x)\cdot f (t,x)=0, \text{ for all } x\in M, t\geq 0$$
is strongly invariant under the flow of the stochastic system \eqref{IE} in the Stratonovich sense, if and only if, 
$$\nabla F(x)\cdot \sigma (t,x)=0, \text{ for all } x\in M, t\geq 0.$$

We have by construction that  $\mu . \left [ \mu \wedge b \right ] =0$ and $\mu . \left [ \mu \wedge \left [ \sigma (t,\mu ) dW_t \right ] \right ] =0$ for any $\sigma$ so that the previous assumptions are satisfied for the sphere $S^2$. As a consequence, the sphere $S^2$ is invariant. This concludes the proof.
\end{proof}

The behavior of the equilibrium point is as usual different:

\begin{lemma}
\label{larmorequilibrium}
The equilibrium points $\pm b/\parallel b\parallel$ persist for the stochastic Larmor equation if and only if the diffusion term $\sigma (t,\mu) dB_t$ is of the form 
\begin{equation}
b. \gamma (\mu) dB_t ,
\end{equation}
where $B_t$ is one dimensional and $\gamma :\R^3 \rightarrow \R$ is a scalar function.
\end{lemma}

\begin{proof}
Moreover, using Lemma \ref{persistenceexternalequilibrium}, we deduce that a stochastic Larmor equation preserves the equilibrium points if and only if $\sigma (t,\mu ) dB_t$ belongs to the Kernel of $p\rightarrow f(\pm \di\frac{b}{\parallel b\parallel} ,p)$. As $f(\pm \di\frac{b}{\parallel b\parallel} ,p) =\pm\di\frac{1}{\parallel b\parallel} b\wedge p$ and $b\wedge p=0$ if and only if $p$ is colinear with $b$, we deduce that $\sigma (t,x) dB_t$ must be of the form  $b\, \gamma (\mu ) dB_t$ where $B_t$ is one dimensional and $\gamma (\mu)$ is a scalar function. This concludes the proof.
\end{proof}

One can go even further by looking for the persistence of the first integral given by $V(\mu )=\mu .b$:

\begin{lemma}
The first integral $V(\mu )=\mu .b$ is not preserved by a stochastic Larmor equation unless the diffusion term is such that 
\begin{equation}
b\wedge [\sigma (t,\mu ) d_{\circ} W_t ] =0,
\end{equation}
for all $t\in \R$ and $\mu \in S^2$.
\end{lemma}

\begin{proof}
The It\^o formula gives
\begin{equation}
\left .
\begin{array}{lll}
d[V(\mu_t )] & = & d\mu .b =[(\mu \wedge b).b ] dt + \epsilon (\mu \wedge \left [ \sigma (t,\mu ) d_{\circ} W_t \right ]).b ,\\
 & = & \epsilon (\mu \wedge \left [ \sigma (t,\mu ) d_{\circ} W_t \right ]).b .
\end{array}
\right .
\end{equation}
We have $(\mu \wedge \left [ \sigma (t,\mu ) d_{\circ} W_t \right ]).b =0$ if and only if the three vector $\mu$, $b$ and $\sigma (t,\mu ) d_{\circ} W_t$ are coplanar. As this equality must be satisfied for all $\mu \in S^2$, this implies that the vector $\sigma (t,\mu ) d_{\circ} W_t$ is collinear to $b$ or equivalently that $b\wedge [\sigma (t,\mu ) d_{\circ} W_t ] =0$. This concludes the proof.
\end{proof} 

The previous result implies that the conservative nature of the Larmor equation is generically lost under a Stratonovich external perturbations. This is in particular the case when $\sigma (t,\mu )$ is a constant vector as one can see in the following simulations. 

\subsection{Simulations}

All the simulations are done under the following set of initial conditions :
\begin{itemize}
	\item Period of time $T=10$,
	\item Bound $\varphi=\frac{\pi}{10}$,
	\item Perturbation coefficient $\sigma_1=1$,
	\item Perturbation coefficient $\sigma_2=5$.
\end{itemize}
In order to do numerical simulations we use the Euler-Murayama scheme with time-step $h=10^{-4}$. We perform the simulations for three initial conditions of $\mu_t$ and each one with new realizations of Brownian motion $B^\theta_t$ and $B^\alpha_t$. \\

In Figure 1, we display the simulation of a stochastic Larmor equation for different values of the initial condition $\mu_0 \in S^2$. The green line corresponds to the dynamics of the deterministic Larmor equation  where the Larmor precession is clearly visible. The red one to the behavior of some solutions of the stochastic Larmor equation near the North or South pole. The solutions stay in the two cases on the sphere $S^2$ and we can see that the deterministic solutions belongs to circle obtain by taking the intersection between the sphere and the plane passing through $\mu_0$ and orthogonal to $b$. In the stochastic case, these circles are broken which traduces the fact that the first integral $V$ is not preserved under the stochastisation procedure.\\

\begin{center}
\begin{figure}[!ht]
\centerline{%
    \begin{tabular}{cc}
        \includegraphics[width=0.4\textwidth]{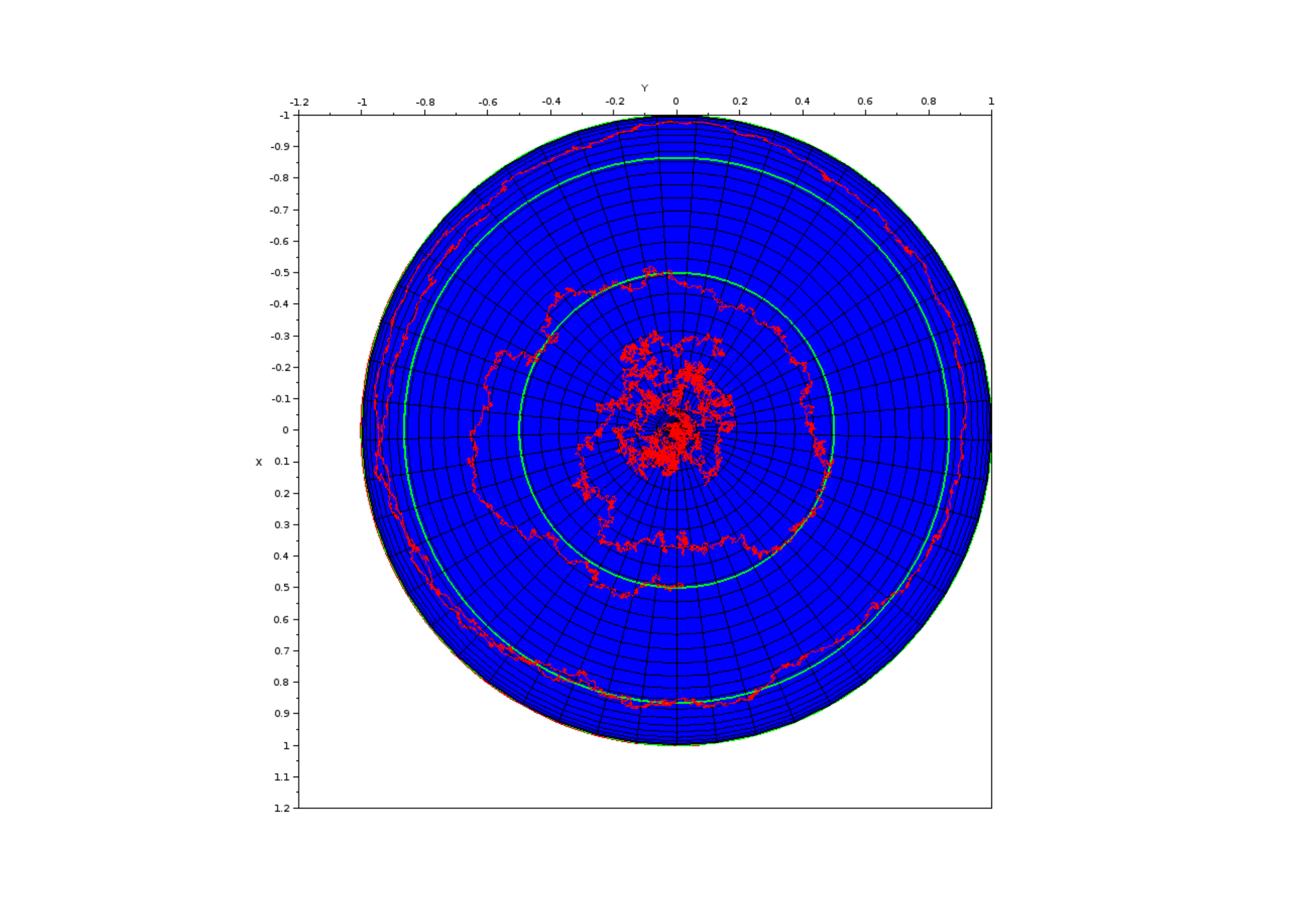} & \includegraphics[width=0.4\textwidth]{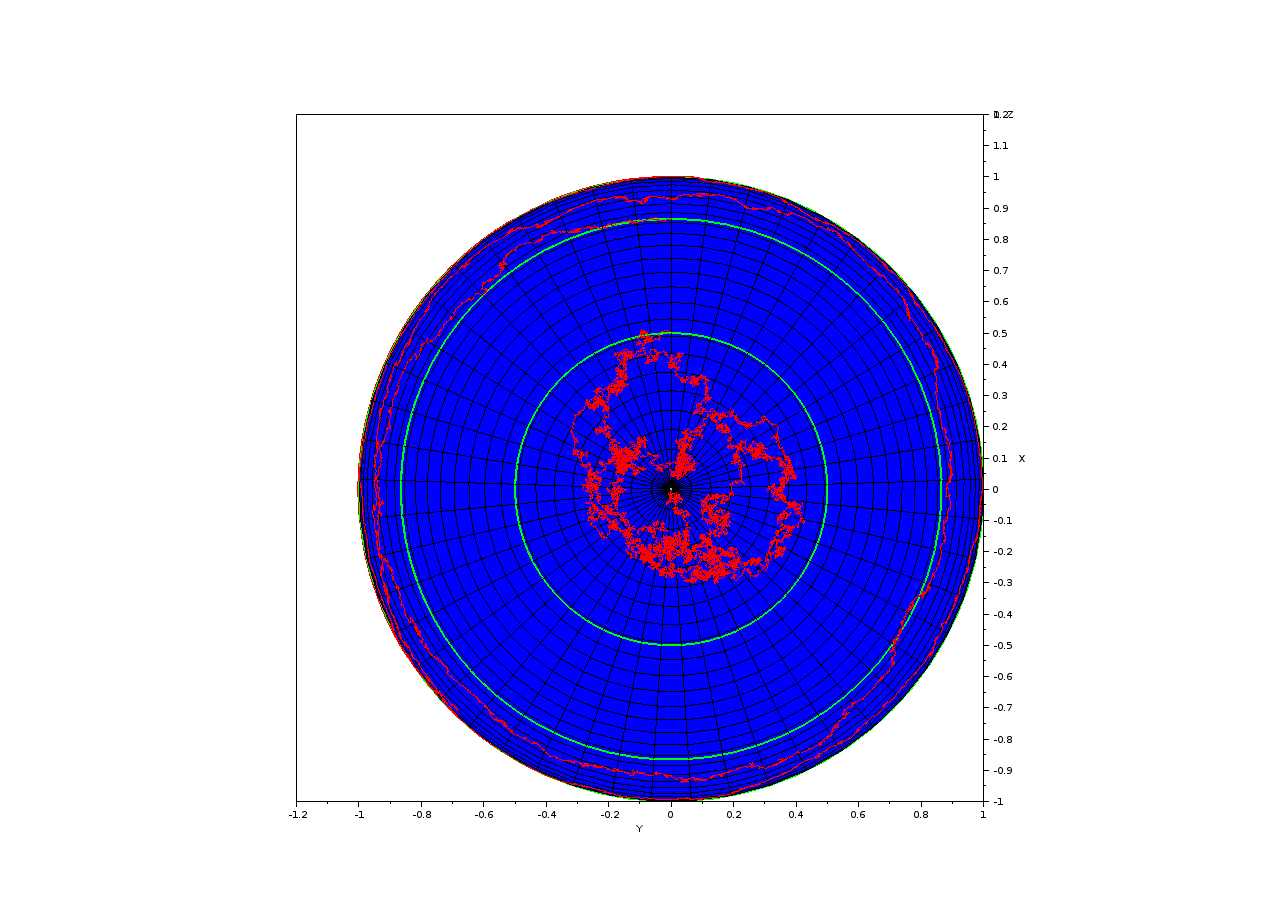}\\
        \footnotesize{(a) Dynamics near the North pole.} & \footnotesize{(b) Dynamics near the South pole}\\
    \end{tabular}}
    \caption{Simulation of a stochastic Larmor equation}
\end{figure}
\end{center}

A global view is given by:

\begin{figure}[!ht]
	\centering
	\includegraphics[width=0.7\textwidth]{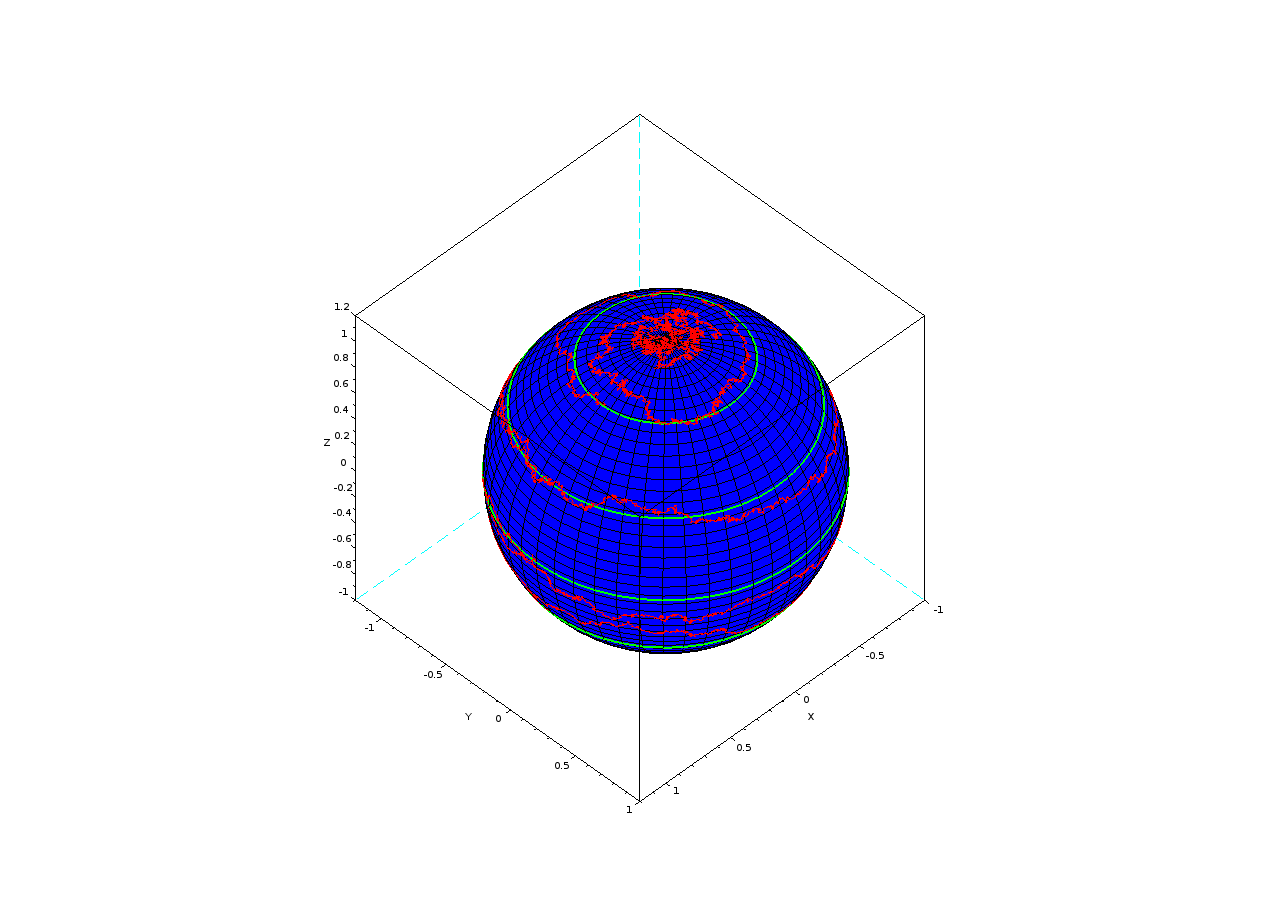}
	\caption{Global dynamics of a stochastic Larmor equation}
	\label{ex1fig3larmor}
\end{figure}

Of course, conservative systems are very sensitive to perturbations in general and one must be very careful in order to preserve this property. The Landau-Lifshitz equation, which can be seen as a perturbation of the Larmor equation by a {\bf dissipative} term, will behave very differently under perturbations. This is due in particular to the robustness of the equilibrium points under perturbations.

\section{The RODEs stochastisation procedure}
\label{rodeapproach}

In this Section, we first remind the definition of a {\bf random ordinary differential equation} (RODE) and we define the RODE stochastisation procedure. We then discuss the persistence of invariance in this framework.

\subsection{The RODEs stochastisation procedure}

For this section, we refer to sevral works of P.Kloeden (\cite{Kloeden}) and in particular the book (\cite{kloeden2}).\\

Up to now, the source of stochasticity was considered as {\bf external} like the modification of the environment in which the dynamical process takes place. However, in many situations, the stochasticity comes from the dynamical behavior of some parameters which were previously considered as fixed. This situation refers to {\bf internal} sources of stochasticity. This is for example the case in biology, when considering the dynamic of neurons where the opening and closing of an ion channel at an experimentally fixed membrane potential is known to be {\bf intrinsically} stochastic (see \cite{Saarinen}, \cite{Fox}).\\

The natural framework to deal with is to consider a {\bf parameters family of differential equations} denoted by $(E_{\eta} )_{\eta \in \R^n}$ and defined by 
$$
\frac{dx}{dt}=f(x,\eta),\quad x\in \R^d,
\eqno{(E_{\eta} )_{\eta \in \R^d}}
$$
where $\eta \in \R^d$ is a set of parameters and $f :\R^n \times\R^d  \rightarrow \R^n $ is a  continuous function.\\

Taking into account stochasticity can be made using the theory of {\bf random ordinary differential equations}, denoted simply by RODE's system in the following which are promoted by P.Kleoden in (\cite{kloeden2}).

\begin{definition}
Let $(\Omega ,\mathcal{F}, P)$ be a probability space, where $\mathcal{F}$ is a $\sigma$ algebra on $\Omega$ and $P$ is a probability measure. Let $\eta_t(\omega)$ be a $\R^d$ valued stochastic process with continuous sample paths. A $\eta_t$-RODE version of the family $(E_{\eta} )_{\eta \in \R^d}$ is the family of autonomous equations 
$$
\frac{dx}{dt}=f(x,\eta_t(\omega))
\eqno{(E_{\eta_t} )}
$$
for almost every realization $\omega \in \Omega.$
\end{definition}

The previous way to associate a RODE's equation to a parameter dependent differential equation will be called {\bf RODE's stochastisation} in contrast with the classical {\bf stochastisation} procedure remind in Section \ref{remindstoc} and applied to parameter dependent systems in Section \ref{stochastisationparameter}.\\

A typical example is given by the scalar RODE equation
\begin{equation}
\di\frac{dx}{dt} =-x+\sin (W_t (\omega )),
\end{equation}
where $W_t (\omega )$ is a standard Wiener process. In this case, we have $f(x, \eta )=-x+\sin (\eta )$ and $n=d=1$.\\

The main point is that the stochasticity is precisely related to the choice of the stochastic dynamics for the parameters. This dynamics is the only one which can be experimentally studied (see \cite{Saarinen}, \cite{Fox}).\\

Many results exist with a very general noise called {\bf RODEs with canonical noise}. We prefer here, to restrict our attention to noise which can be modeled using classical stochastic differential equations which is close to the modeling process. In particular, we consider It\^{o}'s version of RODEs, i.e., RODEs equations of the form 
\begin{equation}
\left\lbrace \begin{array}{l}
\frac{dx}{dt}=f(x,\eta_t(\omega)) \\
d\eta_t=g(t,\eta_t)dt + \sigma(t,\eta_t)dW_t
\end{array}
\right.
\end{equation}
We then recover a special class of It\^{o}'s differential equation for which our results of Part I apply.

\begin{remark}[Existence and uniqueness of solutions] If the initial value $x_0$ is an $\R^d$-valued random variable, has a unique pathwise solution $x(t,\omega)$ for every $\omega \in \Omega,$ which will be assumed to exist on the finite time interval $[0,T] .$ Sufficient conditions that guarantee the existence and uniqueness of such solutions are similar to those for ODEs given on \cite[ch2.1]{kloeden2}, that the vector field is at least continuous in both of its variables and the sample paths of the noise $\eta_t$ are continuous. If it satisfies a global Lipschitz condition, then existence on the entire time interval is obtained.
\end{remark}

\subsection{Equilibrium points of RODE's stochastised systems}

We recall that the stochastisation of a linearly parameter dependent differential equation possesses the same set of equilibrium points as the deterministic systems independently of the It\^o or Stratonovich equation. This property fails generically for an arbitrary equation and stochastic perturbation. The RODE's stochastisation is very particular with respect to this property. Indeed, we have:

\begin{lemma}[Persistence of Equilibrium-RODE's case] Let us assume that the sto\-chastic process $\eta_t$ takes its values in some parameter set $\mathcal{P} \subset \R^n$. Such a stochastic process is said $\mathcal{P}$-valued process. We denote by $\mathcal{E}_{\mathcal{P}}$ the intersection of all the set of equilibrium points of the deterministic system for all $p\in \mathcal{P}$. The equilibrium set $\mathcal{E}_{\mathcal{P}}$ is preserved under any $\mathcal{P}$-valued RODE's stochastisation.
\end{lemma}

\begin{proof}
Let $p_0\in \R^n$ be given. The set of equilibrium points of the deterministic system for $p=p_0$ is denoted $\mathcal{E}_{p_0}$. A point $x_{\star} \in \mathcal{E}_{p_0}$ satisfies the equation $f(x_{\star} ,p_0 )=0$. Assume that $x_{\star} \in \mathcal{E}_{\mathcal{P}}$, i.e. that $x_{\star} \in \mathcal{E}_{p_0}$ for all $p_0 \in \mathcal{P}$ then $f(x_{\star} ,p_0 )=0$ for all $p_0 \in \mathcal{P}$. As $\eta_t (\omega ) \in \mathcal{P}$ for all $t\in \R$ and $\omega \in \Omega$, we deduce that if $x_{\star} \in \mathcal{E}_{\mathcal{P}}$ then $f(x_{\star} ,\eta_t (\omega ) )=0$ for all $t\in \R$ and $\omega \in \Omega$. As a consequence, the set $\mathcal{E}_{\mathcal{P}}$ is preserved under any $\mathcal{P}$-valued RODE's stochastisation.
\end{proof}

The previous Theorem is very general and the condition to find a common equilibrium point for a large class of parameters seems to be out of reach. However, a classical example where such a set is not empty concerns linear systems with a specific dependence on the parameter:

\begin{lemma}
Let us assume that $f(x,p)$ is linear with respect to $x$. Then, for each $p\in \R^n$, equilibrium points of the deterministic system belongs to the vector space $\mathcal{E}_p$ defined by the kernel of $f(x,p)$. If there exists an open set $\mathcal{P}\subset \R^n$ such that $\mathcal{E}_p =\mathcal{E}_{p'}$ for all $p,p'\in \mathcal{P}$ then $\mathcal{E}_{\mathcal{P}} =\mathcal{E}_p$ for any $p\in \mathcal{P}$.
\end{lemma}

We omit the proof which follows easily from the assumptions. One can wonder if we can find explicit examples where this result can be of some used.

\subsection{RODEs stochastisation of the Larmor equation}
\label{larmor1}

The Larmor equation possesses $S^2$ as an invariant manifold and two equilibrium point by restruction on $S^2$ given by $\pm b/\parallel b \parallel$. Using the previous results, one obtain:

\begin{lemma}
\label{equilibriumlarmor}
Let $b_0 \in \R^3$ be fixed. Let $\eta_t$ be a stochastic process in $\R^3$ such that $\eta_t (\omega ) \in \langle b_0 \rangle$. Then any $\langle b_0 \rangle$ valued $\eta_t$ RODEs stochastisation of the Larmor equation possesses as a set of equilibrium points 
\begin{equation}
\mathcal{E}_{\langle b_0 \rangle} =\langle b_0 \rangle .
\end{equation}
\end{lemma}

It must be noted that up to now, we do not restrict the flow of the RODE stochastized on the sphere $S^2$ as we have not proved for the moment that this restriction keep sense in the stochastic case. We will return to this problem in the next Section.

\begin{proof}
Let $b_0 \in \R^3$ fixed and  $\eta_t$ be a $\langle b_0 \rangle$-valued stochastic process. We do not provide an explicit equation for $\eta_t$. Let us consider the $\eta_t$-RODE Larmor equation
\begin{equation}
\di\frac{d\mu}{dt} = \mu \wedge \eta_t .
\end{equation}
As $\eta_t$ is $\langle b_0 \rangle$ valued for any $\mu \in \langle b_0 \rangle$ we have 
\begin{equation}
\mu \wedge \eta_t (\omega ) =0,\ \ \forall \omega \in \Omega\ \mbox{\rm and}\ t\in \R .
\end{equation}
As a consequence, any point in $\langle b_0 \rangle$ is an equilibrium point of the $\eta_t$-RODE Larmor equation. Moreover, a point $\mu\in \R^3$ is an equilibrium point of the $\eta_t$-RODE Larmor equation if and only if $\mu \wedge \eta_t (\omega )=0$ for all $t\in \R$ and $\omega \in \Omega$. As $\eta_t (\omega )\in \langle b_0 \rangle$ this implies that $\mu \in \langle b_0 \rangle$. This concludes the proof.
\end{proof}

The previous Lemma has strong consequences on the possible models one can construct as long as one want to preserve the set of equilibrium points. Indeed, one has to construct the $\eta_t$ process such that it is $\langle b_0 \rangle$ valued. We will find this problem again about the construction of a stochastic Landau-Lifschitz equation in Section \ref{landau}.
 
\subsection{Invariance using RODE's}

The main property of RODEs version of parameters family of ordinary differential equations having an invariant manifold is that this manifold is preserved under the stochastisation process. 

\begin{theorem}[Persistence of invariance-RODEs case]
\label{persistencerode}
Let us consider a parameters family of equations of the form $(E_{\eta} )_{\eta\in \R^d}$ such that for any $\eta\in \R^d$ the manifold 
$$M=\{ x\in \R^n : F(x)=0\}$$
is invariant.
Then, for any $\eta_t$ stochastic process, the $\eta_t$-RODE version of $(E_{\eta} )_{\eta\in \R^d}$ preserves invariance of the manifold $M$.
\end{theorem}

\begin{proof}
Fix a sample path, i.e., we look at our equation as a deterministic equation for each $\omega \in \Omega$. Then we can study the invariance of the submanifold $M$ by studying the derivative of $F(x_t)$ in the usual sense. We obtain that 
\begin{equation}
\di\frac{d}{dt} (F(x_t )) =\nabla F (x_t ) . \di\frac{dx_t}{dt}=\nabla F (x_t ).f(x_t ,\eta_t ).
\end{equation}
As $M$ is invariant for the family of differential equations $(E_{\eta} )_{\eta\in \R^d}$, we have 
\begin{equation}
\nabla F(x ). f(x ,\eta )=0\ \forall x\in M,\ \eta \in \R^d .
\end{equation}
We deduce that $\di\frac{d}{dt} (F(x_t ))=0$ for all solutions of the RODE equation. As a consequence, the manifold $M$ is invariant.
\end{proof}

This Theorem is in strong contrast with our previous result on {\bf external It\^o perturbations} of differential equations where invariance is generically destroyed. More or less, one can say that {\bf internal It\^o perturbations} modeled using RODEs framework preserve invariance and external one destroy it. As a consequence, we see that the precise nature of the noise in the modeling is fundamental.

\subsection{Invariance of $S^2$ for RODEs Larmor equations}
\label{larmor2}

We return on the example studied in Paragraph \ref{larmor1}. Using the previous Theorem, we easily have:

\begin{lemma}[Invariance of $S^2$ for RODEs Larmor equations] 
\label{invariancelarmor} 
Any RODEs stochastisation of the Larmor equation preserves the invariance of $S^2$.
\end{lemma}

Using this result, we can precise the result about equilibrium points obtained in Paragraph \ref{larmor1}:

\begin{lemma}
\label{classlarmor}
Let $b_0 \in \R^3$ be fixed. Let $\eta_t$ be a stochastic process in $\R^3$ such that $\eta_t (\omega ) \in \langle b_0 \rangle$. Then any $\langle b_0 \rangle$ valued $\eta_t$ RODEs stochastisation of the Larmor equation restricted to $S^2$ possesses as a set of equilibrium points 
\begin{equation}
\mathcal{E}_{\langle b_0 \rangle} =\left\{ \pm b_0 /\parallel b_0 \parallel \right \} .
\end{equation}
\end{lemma}

\begin{proof}
By Lemma \ref{invariancelarmor}, if $\mu_0 \in S^2$ then $\mu_t \in S^2$ for all $t\in \R$. As a consequence, the restriction of any $\eta_t$-RODE stochastisation of the Larmor equation has a sense. Let us then consider the restriction of a $\langle b_0 \rangle$ valued $\eta_t$ RODE Larmor equation. By Lemma \ref{equilibriumlarmor}, we deduce that points in $\mathcal{E}_{\langle b_0 \rangle} \cap S^2$ are the only equilibrium points of the RODE Larmor equation. This concludes the proof.
\end{proof}

As a consequence, one can easily find a class of RODE stochastized Larmor equations preserving the essential features of the deterministic equation. We will see if some essential qualitative behavior of the equation as for example the stability nature of the equilibrium points or the asymptotic behavior of the solutions can be preserved under some RODE stochastisation.

\section{Stochastic persistence: stability}
\label{stabilitypart}

We study the persistence problem for what concerns the stability/instability of equilibrium points. As we will see, the behavior of the stochastically perturbed system will depend on the internal or external nature of the noise. We begin by recalling some well known results on stable/unstable equilibrium point for deterministic systems and in particular  Lyapunov theory. We then recall stochastic analogue of these notions. We finally give persistence results for external or internal perturbations. 

\subsection{Stability in the deterministic case}
\label{remindstabilitydeter}

\subsubsection{Stable and asymptotically stable equilibrium}

Let us consider an autonomous ordinary differential equation of the form
$$
\di\frac{dx}{dt} =f(t,x),\ t\in \R,\ x\in \R^n .
\eqno{(E)}
$$
An {\bf equilibrium point} or {\bf steady state} for (E) is a point $x_0 \in \R^n$ such that $f(t, x_0 )=0$ for all $t\in \R$.\\

Stability of a dynamical system means insensitivity of steady states of a system to small changes in the initial states. We are leaded to the following definition of a {\bf stable} equilibrium point:

\begin{definition}
The equilibrium solution $x_t=0$ of \eqref{DE} is said to be stable, if for every $\varepsilon > 0,$ and $t _{0}$ there exists $r=r(\varepsilon,t _{0})>0,$ such that  
$$ \sup_{t\geq t _{0}}\vert x_{t}(x_0)\vert \leq \varepsilon $$
whenever $\vert x_0\vert< r.$ Otherwise, it is said to be unstable.
\end{definition}
 
A more stronger notion is given by the {\bf asymptotic stability} of an equilibrium point. 
 
\begin{definition}
The equilibrium solution $x_t=0$ of \eqref{DE} is said to be asymptotically stable if it is stable, and 
$$ \lim_{t\rightarrow \infty} x_t(x_{0}) =0 $$
for all $x_0$ in some neighborhood of $x=0.$ 
\end{definition} 

A very useful characterization of stable and unstable equilibrium points is given thanks to {\bf  Lyapunov functions}. We remind some classical results in the following Section.

\subsubsection{Lyapunov functions and stability}
 
The idea in Lyapunov stability is that if one can choose a function $V=$ constant represents tubes surrounding the line $x=0$ such that all solutions cross through the tubes toward $x=0,$ the system is stable, while if they cross in the other direction the system is unstable. 

\begin{definition}[ Lyapunov function] 
Let assume that there exist a positive definite function $V(t,x),$ that is continuously differentiable with respect to $x$ and to $t$ throughout a closed neighborhood $U$ of $x=0$, such that $V(t,0)=0$ for all $t\geq t_0$. This function will be called a  Lyapunov function.
\end{definition}

It must be noted that the derivative of $V(t,x_t )$ with respect to $t$ along a given solution of the equation \eqref{DE} is given by 
\begin{equation}
\dot {V} (t,x)= \frac{d V}{d t} = \frac{\partial V}{\partial t}+\sum_{i=1}^n f_i(t,x)\frac{\partial V}{\partial x_i} .
\end{equation}

Depending on the properties of the  Lyapunov function, one obtain stable or asymptotically stable equilibrium points.
 
\begin{theorem}
Let us assume that there exist $V(t,x)$ a positive definite  Lyapunov function in a neighborhood $U$ of the set $x=0$. Then, we have:
\begin{itemize}
\item If  $\dot{V}\leq 0$ for all $x\in U$ and $t\geq 0,$ then the equilibrium solution of the equation \eqref{DE} is stable.
\item If the function $V$ has an arbitrarily small upper bound and $\dot{V}$ is negative definite, then the equilibrium solution of the equation \eqref{DE} is asymptotically stable. 
\end{itemize}
\end{theorem}

When $V$ is independent of the time variable $t$, the previous criterion reduces to the classical {\bf geometric criterion}: if $\Delta V .f \leq 0$ meaning that for $c_0 >0$ sufficiently small the vector field associated to the equation \eqref{DE} points always toward the inside of the area whose boundary is given by $V =c_0$ and containing the equilibrium point $0$ or is tangent to this boundary.\\

\subsubsection{Lyapunov functions and instability}

A function is said to have an arbitrarily small upper bound if there exists a positive definite function $u(x)$ such that
$$V(t,x)\leq u(x), \text{ for all } t\geq t_0.$$
For the instability, let us reminder only the first Lyapunov theorem on instability where the technique in the second theorem is similar to the first such that $dV=\lambda V+W;\; \lambda > 0$ and $W$ is either identically zero or satisfies the conditions of the first theorem of Lyapunov. (see \cite[p 9]{Krasovskii})
 
\begin{theorem}
 Assuming that there exist a function $V$ for which the total derivative associated to the equation \eqref{DE} is definite, assume that $V$ admits an infinitely small upper bound, and assume that for all values of $t$ above a certain bound there are arbitrarily small values $x_s$ of $x$ for which $V$ has the same sign as its derivative, then the equilibrium is unstable.
\end{theorem}

\subsection{Stochastic stability}

In stochastic systems, it is possible to defined a several types of stochastic stability corresponding to the deferents types of convergence of a stochastic process as stability in probability, $p$-stability, exponential $p$-stability, and the stability almost surely in any of the last senses.

\subsubsection{Definitions}

In this paper we are interested to the stability in probability and asymptotic stability in probability, which be called stochastic stability, where the technique in this case is similar to the Lyapunov stability of the deterministic case.\\

If we assume that $f(t,0)=\sigma(t,0)=0,$ for each $t\geq 0,$ then $x_t(0)=0$ is the unique solution of the stochastic system\eqref{IE} with initial state $x_0=0$ which will be called the equilibrium solution, whose stability is be examined. 
For more details, we refer to the books by Khasminskii \cite{Khasminskii} and Arnold \cite[p.179-184]{Arnold}.

\begin{definition}
 The equilibrium solution of \eqref{IE} is said to be stochastically  stable, if for every $\varepsilon >0,\delta >0,$ and $t _{0}$ there exists $r=r(\varepsilon,\delta,t _{0})>0,$ such that for every initial value $x_0$ at time $t_0 $ with $\vert x_0\vert< r$ the solution $x_t(x_0)$ exist for all $t\geq 0$ and satisfies
     $$ P\lbrace \omega :  \sup_{t\geq t _{0}}\vert x_{t}(x_0)\vert >\delta \rbrace<\varepsilon .$$
 Otherwise, it is said to be stochastically unstable.
\end{definition}

\begin{definition}
The equilibrium solution of \eqref{IE} is said to be stochastically asymptotically stable if it is stochastically stable, and if furthermore, for every $\varepsilon$ and $t _{0}$ there exist $r(\varepsilon,t _{0})>0$ such that
 $$\lim_{x_0\rightarrow 0} P \lbrace \omega : \lim_{t\rightarrow \infty} x_t(x_{0}) =0 \rbrace \geq1-\varepsilon, \;\text{ for all } \vert x_0\vert< r.$$
\end{definition}

If in addition to being stable in the (resp. asymptotically stable in probability), the equilibrium has the property that for any finite initial state, then it is globally stable (resp, globally asymptotically stable) in the large.

\subsubsection{Lyapunov functions and stability}

Let $V(t,x)$ a positive-definite function that is twice continuously differentiable with respect to $x,$ continuously differentiable with respect to $t$ throughout a closed domain $U$ in $\mathbb{R}^{+}\times \mathbb{R}^{n},$  except possibly for the set $x=0,$ and equals to zero only at the equilibrium point $x=0.$ The generator operator $L$ of the process defined by the equation \eqref{IE} for any function $V\in C^{2}$ is given as 
$$ LV (t,x)= \frac{\partial V}{\partial t}+\sum_{i=1}^n f_i(t,x)\frac{\partial V}{\partial x_i}+ \frac{1}{2}\sum_{i,j=1}^n  \frac{\partial^2 V}{\partial x_i \partial x_j}\sum^{l}_{k=1}\sigma_{i,k}(t,x_{t})\sigma_{j,k}(t,x_{t}),$$ 
where the derivative of $V$ under It\^o formula is given by
\begin{equation}
\label{dv}
 dV_t= LV(t,x)+ \sum_{i=1}^n\sum_{j=1}^l \frac{\partial V}{\partial x_i}\sigma_{i,j}(t,x_{t})dW^{j}_t 
\end{equation}  
A stable system should have the property that $V_t$ does not increase, that is $dV_t\leq 0.$ This would mean that the ordinary definition of stability holds for each single trajectories $x(\omega).$ However, because of the presence of the fluctuation term in \eqref{DE}, the condition $ dV_t\leq 0$ is formulated as $ \mathbb{E}[dV_t]\leq 0,$ which reduces to the form $ dV_t\leq 0$ in the deterministic case if $\sigma$ vanishes. Since
$$\mathbb{E}[dV_t]= LV_t dt,$$
the condition in the Lyapunov theorem is given as 
$$ LV_t(t,x)\leq 0, \quad \text{for all } t\geq t_0, x\in \mathbb{R}^{n}.$$ 

\begin{theorem}
\label{stability}
Let us assume that there exist a  Lyapunov function $V(t,x)$ in a neighborhood $U$ of the set $x=0$. Then, we have:
\begin{itemize}
\item If $LV \leq 0$ for all $x\in U$ and $t\geq 0,$ then the equilibrium solution of \eqref{IE} is stochastically stable.
\item If in addition, $V$ has an arbitrarily small upper bound and $LV$ is negative definite, then the zero solution is stochastically asymptotically stable.
\end{itemize} 
\end{theorem}

\subsubsection{Lyapunov functions and instability}

The analogs of the instability theorem of Lyapunov do not hold in the stochastic systems. Thus we will reminder only the following theorem of Khasminskii \cite{Khasminskii} which will be used to prove the instability of one of the equilibriums of the stochastic Landau-Lifshitz equation.

\begin{theorem}
\label{instability}
Assuming that there exist a function $V(t,x)$ that is twice continuously differentiable with respect to $x$  throughout a subset $U_r={\vert x\vert<0},$ continuously differentiable with respect to $t,$ such that
 $$\lim_{x\rightarrow 0}\inf_{t>O} V (t,x)= \infty,$$
 and 
  $$\sup_{\varepsilon<\vert x\vert < r}LV < 0, \text{ for any } \varepsilon >0 \text{ as } x\in U_r\text{ and } x \neq 0,$$
 Then the equilibrium solution of \eqref{IE} is stochastically unstable.
\end{theorem}

\begin{example}
Let us consider the scalar linear stochastic differential equation
\begin{equation}
dx_t=ax_t dt + bx_tdW_t, \quad x_{t_0}=c,
\end{equation}
where $a,b,c \in \R$ and $W_t$ is a 1-dimentional Brownian motion.\\

The zero solution is an equilibrium to this equation where the Lyapunov function is given by $V(x)= \vert x\vert^r$ such that $r> 0,$ and its derivative along the process $x_t$ give us  
$$ L \vert x\vert^r =[a+\frac{1}{2}b^2(r-1)]r\vert x\vert^r.$$
As long as $a< \frac{b^2}{2},$ we can choose $r$ such that $0<r<1-2\frac{a}{b^2}$ and hence satisfy the condition 
$$LV\leq -kV, \text{ where }  k>0$$
which, according to theorem \eqref{stability} is sufficient for stochastic asymptotic stability.\\

If $a\geq \frac{b^2}{2}>0,$ we choose $V(x)=-\log\vert x\vert$ and we obtain
$$LV=-a+ \frac{b^2}{2}\leq 0,$$
for $b\neq 0;$ we have the stochastic instability of the zero solution from the theorem \eqref{instability}, for $b=0$ the system is deterministic, where we have the asymptotic stability for $a<0,$ simple stability for $a=0,$ and instability for $a>0.$
\end{example}

The Stratonovich integral lacks an important property of the It\^o integral to generalize the deterministic Lyapunov theorem to the stochastic case, which does not semimartingale, because of that we will study the stability of the solution of \eqref{SE} by studying its modified It\^o equation \eqref{ito}.

\subsection{Persistence of stability - the It\^o case}

We now look from a perturbative point of view the question of the persistence of the stability of an equilibrium point, i.e. we assume that the deterministic system possesses a stable equilibrium point and we look if for a sufficiently small external perturbation, meaning that all the components of $\sigma$ are sufficiently small, this property persists. 

\begin{theorem}[Persistence of stability- It\^o case] 
\label{persistencestabilityito}
Let us assume that $0$ is a stable (resp. asymptotically stable) equilibrium point for a deterministic system such that there exists a Lyapunov function $V(t,x)$ in a neighborhood $U$ of the set $x=0$ such that the Hessian matrix of $V$ is negative definite at $0$. Assume moreover, that the external perturbation of this system possesses $0$ as an equilibrium point. Then for $\sigma$ sufficiently small, the stochastic perturbation preserves the nature of the equilibrium point.
\end{theorem}

\begin{proof}
Using the assumptions on the Hessian matrix, one concludes that $0$ is a maximum for $V$ in a sufficiently small neighborhood of $0$, that is the second order term 
$$\frac{1}{2}\sum_{i,j=1}^n  \frac{\partial^2 V}{\partial x_i \partial x_j}\sum^{l}_{k=1}\sigma_{i,k}(t,x)\sigma_{j,k}(t,x),$$
is negative for $\sigma$ sufficiently small. As a consequence, if $0$ is a stable equilibrium point for the deterministic system then $0$ is again stable in the stochastic case. The same argument applies to the asymptotically stable case.
\end{proof}

\subsection{Persistence of stability - the RODEs case}

In the internal case, we will assume that the RODE system associated to the given deterministic equation is of the form
\begin{equation}
\label{fluctuat}
\left\lbrace \begin{array}{l}
\frac{dx}{dt}=f(x,\eta_t(\omega)) ,\\
d\eta_t= \sigma(t,\eta_t)dW_t ,
\end{array}
\right.
\end{equation}
meaning that $\eta_t$ is a fluctuation with respect to the case of constant parameters.\\

In this case, we have the following result:

\begin{theorem}[Persistence of stability- RODE case] 
Let us assume that $0$ is a stable (resp. asymptotically stable) equilibrium point for a deterministic system such that there exists a Lyapunov function $V(t,x;p)$ in a neighborhood $U$ of the set $x=0$ and for all $p\in \R^m$. Assume moreover, that the internal perturbation of this system possesses $0$ as an equilibrium point. If the Hessian of $V$ with respect to $p$ is negative definite then for $\sigma$ sufficiently small, the stochastic perturbation preserves the nature of the equilibrium point. Moreover, if the Lyapunov function does not depend on $p$ then no conditions on the Hessian of $V$ are needed.
\end{theorem}

\begin{proof}
This is a simple computation. Let $V(t,x;p)$ be a Lyapunov function for the deterministic system, meaning that 
\begin{equation}
\di\frac{\partial V}{\partial t} +\nabla V . f \leq 0 ,
\end{equation}
for all $x$ in $U\setminus \{0\}$ and all $p\in \R^m$. Then, we have using the It\^o formula
\begin{equation}
dV(t;X_t ,\eta_t ) = \left [ \di\frac{\partial V}{\partial t} +\nabla V .f \right ] dt+\di\frac{1}{2} dp_t^{T} Hess_p (V) dp_t ,
\end{equation}
where $Hess_p (V)$ denotes the Hessian matrix of $V$ with respect to $p$. It must be note that due to the deterministic character of the equation on $X_t$, we have no second order term due to $X_t$. \\

We already know that the first order term is negative for $x$ in $U\setminus \{0\}$ and all values of $p\in \R^m$. Assuming that the Hessian of V is negative definite over $U\setminus\{ 0\}$ and all values of $p$, we deduce that for $\sigma$ sufficiently small, we have 
\begin{equation}
\sigma^T Hess_p (V) \sigma <0 ,
\end{equation}
which concludes the proof.
\end{proof}

\subsection{A stochastic Larmor equation: first integral and stability}

We return to the study of the stochastic Larmor equation studied in Paragraph \ref{larmor1} and \ref{larmor2}. Let $b_0 \in \R^3$ be fixed such that L$b_0 \not= 0$. We consider the RODE Larmor equation defined by 
\begin{equation}
\di\frac{d\mu}{dt} = \mu \wedge \eta_t ,
\end{equation}
where $\eta_t$ is a $\langle b_0 \rangle$-valued stochastic process. By Lemma \ref{classlarmor}, we know that the sphere $S^2$ is invariant under the flow of this equation and moreover that the flow restricted to the sphere  possesses only as equilibrium point $\pm b_0 /\parallel b_0 \parallel$. \\

In the deterministic case, we have:

\begin{lemma}
The function $V(\mu )=\mu . b_0$, $\mu \in \R^3$, is a first integral of the Larmor equation $\di\frac{d\mu}{dt} =\mu \wedge b_0$.
\end{lemma}

\begin{proof}
We have 
\begin{equation}
\di\frac{d}{dt} \left [ V(\mu_t ) \right ] = \di\frac{d\mu }{dt} .b_0 =\left [\mu \wedge b_0 \right ] . b_0 =0,
\end{equation}
which concludes the proof.
\end{proof}

A natural question is to look for conditions under which such a first integral is preserved under a RODE  stochastisation process. We have:

\begin{lemma}
The first integral $V(\mu )=\mu .b_0$ is preserved by any $\eta_t$-RODE stochastisation where $\eta_t$ is a $\langle b_0 \rangle$ valued stochastic process.
\end{lemma}

\begin{proof}
As $\mu_t$ satisfies a RODE equation, we have by differentiating with respect to $t$ the function $V(\mu_t )$ that 
\begin{equation}
d\left [ V(\mu_t ) \right ] = \di\frac{d\mu}{dt} .b_0 =\left [ \mu_t \wedge \eta_t \right ] .b_0 .
\end{equation}
As $\eta_t$ is a $\langle b_0 \rangle$-valued process, we have
\begin{equation}
\left [ \mu_t \wedge \eta_t \right ] .b_0 =0 .
\end{equation}
This concludes the proof.
\end{proof}

As a consequence, from a qualitative view point, the $\eta_t$ RODE Larmor equation where $\eta_t$ is a $\langle b_0 \rangle$ valued process preserves the main properties of the initial system. The dynamics differs only at a quantitative level. As a consequence, the two equilibrium points are in this case Lyapunov stable:

\begin{lemma}
The two equilibrium points $\pm b_0 /\parallel b_0 \parallel$ of any $\eta_t$-RODE Larmor equation where $\eta_t$ is a $\langle b_0 \rangle$-valued stochastic process are stochastically stable.
\end{lemma}

\begin{proof}
Let $U$ be an open neighborhood of $\pm b_0 /\parallel b_0 \parallel$ on the sphere $S^2$. This neighborhood is foliated by the level surfaces of the function $V(\mu )=\mu .b_0$ which gives concentric circles centered on $\pm b_0 /\parallel b_0\parallel$ on $S^2$. As a consequence, the equilibrium points $\pm b_0 /\parallel b_0\parallel$ are stochastically stable.
\end{proof}

The stability is here obtained from a geometrical argument without using a  Lyapunov type argument.

\section{Stochastic persistence problem - Symplectic and Poisson Hamiltonian structures}
\label{symppoisson}

In this Part, we study the behavior of symplectic and Poisson Hamiltonian systems under stochastic perturbations. These systems have very strong properties. In particular, they can be derived from a variational principle. The persistence of these properties for the stochastic system will play an important role when selecting a stochastic model for a given phenomenon. However, as for stochastic gradient systems, it is not clear how to define a satisfying stochastic analogue of Hamiltonian systems. Several notions already exist and follows different paths. One part take its source in the seminal work of J-M. Bismut in \cite{bismut} defining what he called {\bf random mechanics} and the other one in the work of E. Nelson \cite{nelson} on {\bf stochastic mechanics}. Although the two approaches are related, the resulting definitions are different. For a generalization of Bismut work we refer to \cite{cami} and to a developpment of Nelson's approach we refer to \cite{cresson} and \cite{zambrini}. In this Part, we use the work of C\'ami and Ortega \cite{cami} for what concern the Stratonovich perturbation framework. However, and contrary to \cite{cami}, we discuss alternative definitions in the It\^o case and the RODE case. 

\subsection{Symplectic and Poisson Hamiltonian systems}

\subsubsection{Symplectic structures and Hamiltonian systems}
\label{symplectichamiltonian}

We review classical result on Hamiltonian systems on linear symplectic spaces following the book of J.E. Marsden and T.S. Ratiu (\cite{ratiu},Chap.2). All these part can be extended to manifolds (see \cite{ratiu},Chap.5) but we do not need such a generality in the following. 

\paragraph{Symplectic forms and symplectic spaces}

In the following, $Z$ denote an even dimensional vector space.

\begin{definition}
A symplectic form $\Omega$ on a vector space $Z$ is a non-degenerate skew-symmetric bilinear form on $Z$. The pair $(Z,\Omega )$ is called a symplectic vector space. 
\end{definition}

We often deal with specific symplectic forms called {\bf canonical forms}. Let $W$ be a real vector space of dimension $n$ and let $Z =W\times W^*$. We define the {\bf canonical symplectic form $\Omega$} on $Z$ by 
\begin{equation}
\Omega \left ( (w_1 ,\alpha_1 ) ,(w_2 ,\alpha_2 ) \right ) =\alpha_2 (w_1 ) -\alpha_1 (w_2 ) ,
\end{equation}
where $w_1 ,w_2 \in W$ and $\alpha_1 ,\alpha_2 \in W^*$. \\

Let $\{ e_i \}$ be a basis of $W$ and let $\{ e^i \}$ be the dual basis of $WW^*$. Then the matrix of $\Omega$ in the basis
\begin{equation}
\{ (e_1 ,0) ,\dots, (e_n ,0 ), (0,e^1 ),\dots ,(0,e^n ) \}
\end{equation}
is
\begin{equation}
\mathbb{J} =
\left [ 
\begin{array}{cc}
0 & 1 \\
-1 & 0
\end{array}
\right ] 
,
\end{equation}
where $1$ and $0$ are the $n\times n$ identity and zero matrices.\\

The symplectic form $\Omega$ is then written as 
\begin{equation}
\Omega (z,w) =z \mathbb{J} w^T .
\end{equation}

\paragraph{Hamiltonian function and Hamilton's equations}

\begin{definition}
Let $(Z,\Omega )$ be a symplectic vector space and $H$ an Hamiltonian function. Hamilton's equation for $H$ is the system of differential equations defined by $X_H$, i.e. 
\begin{equation}
\dot{z} =X_H (z) .
\end{equation}
\end{definition}

The {\bf classical Hamilton equation} corresponds to Hamilton's equations in {\bf canonical coordinates} $(q,p) =(q^1 ,\dots ,q^n ,p_1 ,\dots ,p_n )$ with respect to which $\Omega$ has matrix $\mathbb{J}$. In this coordinate system, the Hamiltonian vector field $X_H$ takes the form
\begin{equation}
X_H =\mathbb{J} .\nabla H .
\end{equation}
We then recover the usual form 
\begin{equation}
\left \{ 
\begin{array}{lll}
\di\frac{dq^i}{dt} & = & \di\frac{\partial H}{\partial p_i} ,\\
\di\frac{dp_i}{dt} & = & - \di\frac{\partial H}{\partial q^i} .
\end{array}
\right .
\end{equation}

\subsubsection{Poisson structures and Hamiltonian systems}
\label{poissonhamiltonian}

The previous Section can be generalized on Poisson manifold, which as explained in (\cite{ratiu},Chap.10,p.329), "keep just enough of the properties of Poisson brackets to describe Hamiltonian systems".

\paragraph{Poisson manifolds and Poisson brackets}

We remind some classical properties of Poisson manifolds and Poisson brackets following the book of J.E. Marsden and T. Ratiu \cite{ratiu} to which we refer for more details. 

\begin{definition}
A Poisson bracket (or Poisson structure) on a manifold $P$ is a bilinear operation $\{ .,.\}$ on $\mathcal{F}(P) =C^{\infty} (P)$ such that 
\begin{itemize}
\item $(\mathcal{F}(P) ,\{ . ,.\} )$ is a Lie algebra and $\{ . , .\}$ is a derivation. 
\item $\{ . ,.\}$ is a derivation in each factor, that is
\begin{equation}
\{ FG,H\}=\{F,H\}G +F\{ G,H\} ,
\end{equation}
for all $F,G$, and $H\in \mathcal{F}(P)$.
\end{itemize}
A manifold $P$ endowed with a Poisson bracket on $\mathcal{F}(P)$ is called a Poisson manifold.
\end{definition}

A Poisson manifold is denoted by $(P,\{ .,.\})$.\\
 
Of course, {\bf any symplectic manifold is a Poisson manifold} as noted in (\cite{ratiu},Chap.10,p.330).\\

An important example is given by the {\bf Lie-Poisson bracket}. If $\mathfrak{g}$ is a Lie algebra, then its dual $\mathfrak{g}^*$ is a Poisson manifold with respect to each of the {\it Lie-Poisson brackets} $\{ .,.\}_+$ and $\{ .,.\}_-$ defined by 
\begin{equation}
\{ F,G\}_{\pm} (\mu ) =\pm \langle \mu ,\left [ \di\frac{\delta F}{\delta \mu} ,\di\frac{\delta G}{\delta \mu} \right ] \rangle ,
\end{equation}
for $\mu \in \mathfrak{g}^*$ and $F,G \in \mathcal{F}(\mathfrak{g}^* )$.

\paragraph{Hamiltonian vector fields and Casimir functions}

The notion of Hamiltonian vector field can be extended from the symplectic to the Poisson context. This follows from the following Proposition (see \cite{ratiu}, Chap.10,Proposition 10.2.1 p.335):

\begin{proposition}
Let $P$ be a Poisson manifold. If $H \in \mathcal{F}(P)$, then there is a unique vector field $X_H$ on $P$ such that 
\begin{equation}
X_H [G] =\{ G,H \} ,
\end{equation}
for all $G \in \mathcal{F}(P)$. We call $X_H$ the {\bf Hamiltonian vector field} of $H$.
\end{proposition}

This definition reduces to the symplectic one if the Poisson manifold $P$ is symplectic.\\

A classical differential equation $\dot{z} =f(z)$ will be called an {\bf Hamilton equation} if we can find a Hamiltonian function $H$ such that 
\begin{equation}
\label{hamiltonpoisson}
\dot{z} =X_H (z) .
\end{equation}
This equation can be written in Poisson bracket form as
\begin{equation}
\dot{F} =\{ F,H \} ,
\end{equation}
for any $F\in \mathcal{F}(P)$.\\

The classical transfer between the Poisson structure on $\mathcal{F}(P)$ and the Lie structure on Hamiltonian systems is also preserved (see \cite{ratiu},Proposition 10.2.2 p.335):

\begin{proposition}
The map $H\rightarrow X_H$ of $\mathcal{F} (P)$ to $\mathcal{H}(P)$ is a Lie algebra antihomomorphism; that is 
\begin{equation}
\left [ X_H ,X_K \right ] = -X_{{ H,K}} .
\end{equation}
\end{proposition}

An interesting property is that first integrals of a Hamiltonian systems possess a simple algebraic characterization (see \cite{ratiu},Corollary 10.2.4 p.336):

\begin{corollary}
\label{poissonintegral}
Let $G,H \in \mathcal{F}(P)$. Then $G$ is constant along the integral curves of $X_H$ if and only if $\{ G,H\}=0$. 
\end{corollary}

Among the elements of $\mathcal{F}(P)$ are functions $C$ such that $\{ C,F\}=0$, for all $F\in \mathcal{F}(P)$, that is, $C$ is constant along the flow of all Hamiltonian vector fields or, equivalently, $X_C =0$, that is, $C$ generates trivial dynamics. Such functions are called {\bf Casimir functions} of the Poisson structure. They form the center of the Poisson algebra.

\paragraph{Poisson structure of the Larmor equation}
\label{poissonlarmor}

We first define a Poisson structure on $\R^3$ following (\cite{ratiu},p.331) where the {\bf Rigid Body bracket} is defined for all $z\in \R^3$ by
\begin{equation}
\label{bracketlarmor}
\{ F ,G \}_{\pm} (z)=\pm z.\left ( \nabla F \wedge \nabla G \right ) ,
\end{equation}
where $\nabla F$, the gradient of $F$, is evaluated at $z$.\\

Let us consider the Hamiltonian function defined by 
\begin{equation}
H(z)=z. b ,
\end{equation}
where $b\in \R^3$ is a given constant vector. 

\begin{lemma}
\label{hamlarmorpoisson}
The Hamiltonian vector field $X_H$ defined by $H(z) =z.b$ on the Poisson manifold $(\R^3 ,\{ .,.\}_{\pm} )$ is given for all $z\in \R^3$ by 
\begin{equation}
X_H (z)= -\pm z\wedge b .
\end{equation}
\end{lemma}

\begin{proof}
We have $\nabla H (z)=b$ for all $z\in \R^3$ and 
\begin{equation}
\{ F ,H\}_{\pm} (z)=\pm z. \left ( \nabla F \wedge b \right ) ,
\end{equation}
for all $z\in \R^3$ and $F\in \mathcal{F} (\R^3 , \{ .,.\})$. The mixed product $[u,v,w]=u.(v\wedge w)$ satisfies the relation 
$[u,v,w]=-[v,u,w]$ so that we obtain 
\begin{equation}
\label{equabracket}
\{ F ,H\}_{\pm} (z)=-\pm \nabla F (z). (z \wedge b) .
\end{equation}
By definition of $X_H$, one must have for all $F\in \mathcal{F}((\R^3 , \{ .,.\})$ the equality $X_H [F]=\{ F,H \}_{\pm}$. As $X_H [F] =X_H[z] .\nabla F (z)$, we deduce that 
\begin{equation}
X_H (z)= -\pm z\wedge b .
\end{equation}
This concludes the proof.
\end{proof}

As a consequence, the Hamiltonian equation associated to $H$ on the Poisson manifold $(\R^3 ,\{ .,.\}_{\pm}$ is given by 
\begin{equation}
\dot{z} =- \pm z\wedge b .
\end{equation}
The previous result gives a Poisson Hamiltonian structure to the Larmor equation studied in Section \ref{larmor}:

\begin{lemma}
\label{poissonlarmorlemma}
The Larmor equation (\ref{eqlarmor}) is a Hamiltonian systems with respect to the Poisson structure $(\R^3 ,\{ . ,.\}_{-} )$ with a Hamiltonian function given by $H(z)=z.b$.
\end{lemma}

Using this structure, one can deduce an alternative proof of the invariance of $S^2$ for the Larmor equation. Indeed, we have the following general result:

\begin{lemma}
\label{casimir}
The function $F(z)=\di\frac{\parallel z \parallel^2}{2}$ is a Casimir function of the Poisson manifold $(\R^3 ,\{ .,.\}_{\pm} )$.
\end{lemma}

\begin{proof}
We have $\nabla F (z)=z$. From equation (\ref{equabracket}) we deduce that 
\begin{equation}
\{ F ,G\}_{\pm} (z)=-\pm z. \left ( z \wedge \nabla G (z) \right ) =0.
\end{equation}
This concludes the proof.
\end{proof}

Using Corollary \ref{poissonintegral}, we deduce that:

\begin{corollary}
The sphere $S^2$ is invariant under the flow of the Poisson Hamiltonian $X_H$.
\end{corollary}

\subsection{Stochastic symplectic and Poisson Hamiltonian systems}

\subsubsection{How to select a definition for stochastic Hamiltonian systems ?}

As already discussed in the Introduction of this paper, in many applications one has to choose a stochastic model respecting a certain number of constraints. In this Section, we face a problem of the same nature by searching for a definition of a stochastic object to which some specific features of classical Hamiltonian systems can be attach.\\

The reminder about Hamiltonian systems in Sections \ref{symplectichamiltonian} and \ref{poissonhamiltonian} offers many possibilities:

\begin{itemize}
\item {\bf Algebraic analogue}. We focus on a specific form of Hamiltonian equations in canonical coordinate systems and we just generalize this specific form by extending it directly to the stochastic case. The approach of J-M. Bismut in \cite{bismut} to define stochastic Hamiltonian systems in the Stratonovich setting is an example. We refer to Section \ref{bismutham2} for a discussion. In general, such a generalization is not satisfying as the specific shape of a class of equations is coordinates dependent. As a consequence, one usually look for a stochastic model which preserves a more intrinsic property of the initial system.

\item {\bf Qualitative analogue}. In this case, one select a class of stochastic models by imposing some well known properties like the preservation of the symplectic or volume form (Liouville's theorem) or the conservation of energy. This procedure is reminiscent of our approaches to stochastic gradient equations where the main property of these systems is to possess a global  Lyapunov functions inducing all the qualitative properties of the system. This approach has been used by Milstein in \cite{mil} in order to recover the Bismust stochastic Hamiltonian class as a consequence of the preservation of the symplectic form (see Section \ref{bismutham2}). Qualitative behavior are of course very important but in some cases the constraints imposed by preserving it in the stochastic setting are two strong. This is the case for example if one wants to preserves the conservation of energy (see \ref{bismutham2}).

\item {\bf Geometrical analogue}. Hamiltonian systems possess a rich geometrical structure related to symplectic and Poisson structure which allows to obtain intrinsic/geometric formulations of these equations. An idea is naturally to extend to the stochastic case this geometric framework, thus avoiding the problem of the {\it algebraic analogue} approach. However, preserving geometric properties often mean that one must deal with the Stratonovich approach. Indeed, most of the object and notions which are defined at the differential geometric level are dependent of the specific properties of the differential calculus, in particular the classical chain rule and Leibniz properties. These rules are only preserved in the Stratonovich case which then provide a convenient framework to extend the notions at the stochastic level. This strategy is for example followed by C\'ami and Ortega in \cite{cami}. 

\item {\bf Variational analogue}. One of the main properties of Hamiltonian equations is that they corresponds to critical points of a given functional known as the Hamilton's principle. As this principle does not depend on coordinate systems, we have a very deep characterization of these equations. An idea is then to generalize the functional in a stochastic setting and to develop an appropriate calculus of variation. Here again the Stratonovich is more convenient but the It\^o case has also received much attention due to Nelson's approach to stochastic mechanics (see \cite{nelson},\cite{zambrini},\cite{cresson}).
\end{itemize}

The next Sections discuss in the three frameworks, Stratonovich, It\^o and RODE, the definition of a stochastic analogue of Hamiltonian systems following the previous selection rules.

\subsubsection{Bismut stochastic Hamiltonian systems - Stratonovich case}
\label{bismutham2}

In the Strato\-novich framework, most of the classical differential relation between objects are preserved so that we can manage to define stochastic Hamiltonian in different ways.

\paragraph{Algebraic analogue}

In \cite{bismut}, J-M. Bismut defines stochastic Hamiltonian systems as follows:

\begin{definition}
A stochastic differential equation in $\R^{2n}$ is called stochastic Hamiltonian system if we can find a finite family of functions $\mathbf{H}=\{ H_r \}$, $H_r :\R^{2n} \rightarrow \R$ such that 
\begin{equation}
\left \{
\begin{array}{lll}
dq_i & = & \di\frac{\partial H_0}{\partial p_i} dt +\di\sum_{r=1}^m \di\frac{\partial H_r}{\partial p_i} \circ dB_t^r ,\\
dp_i & = & -\di\frac{\partial H_0}{\partial q_i} dt -\di\sum_{r=1}^m \di\frac{\partial H_r}{\partial q_i} \circ dB_t^r .
\end{array}
\right .
\end{equation}
\end{definition}

We recover the classical algebraic structure of Hamiltonian systems: Let $z=(q,p)\in \R^{2n}$, we have
\begin{equation}
\label{bismutham}
dz = \mathbb{J} \nabla H_0 dt +\di\sum_{j=1}^m \mathbb{J} \nabla H_i \circ dB_t^i .
\end{equation}
The previous property is not satisfying as this specific form depends on the coordinate system. However, Bismut (see \cite{bismut},Chap.V,p.222) proves that we preserve also the Liouville's property and the Hamilton's principle:
\begin{itemize}
\item {Liouville's property}. Let $dz=f(t,z)dt +\sigma (t,z) dB_t$ be of a stochastic differential equation on $\R^{2n}$. The phase flow preserves the symplectic structure if and only if it is a stochastic Hamiltonian system of the form (\ref{bismutham}).

\item {\it Hamilton's principle}. Solutions of a stochastic Hamiltonian system correspond to critical points of a stochastic functional defined by
\begin{equation}
\mathcal{L}_{\mathbf{H}} (X)=\di\int_0^t H_0 (s,X_s ) dt +\di\sum_{r=1}^m H_r (s,X_s ) \circ dB_s^r .
\end{equation}
\end{itemize}

As a consequence, not only the shape of the equation is preserved but intrinsic properties like satisfying an Hamilton's principle.

\begin{remark}
In (\cite{cami}, Section 2, Remark 2.5) the authors pointed out that the choice of the Stratonovich equation is related to fact that "underlying geometric features underlying classical deterministic Hamiltonian mechanics are preserved" in the Stochastic case. However, they say that in the It\^o case it is also possible as we have transfer formula between Statonovic and It\^o stochastic equation. 
\end{remark}

\paragraph{Qualitative analogue}

We have seen that preserving the shape of Hamiltonian systems lead to a notion of stochastic Hamiltonian systems which preserves also the symplectic form. A natural question is then: {\bf assume that the flow of a Stratonovich differential equation preserves the symplectic form. Can we deduce that it is a Bismut stochastic Hamiltonian system ?}\\

This problem was answered positively by G.N. Milstein and al. in \cite{mil} where it is proved the following 

\begin{theorem}
A $2n$ dimensional Stratonovich differential equation possesses a Bismut stochastic Hamiltonian formulation if and only if it preserves the symplectic form.
\end{theorem}

As a consequence, the qualitative approach lead to the same definition as the algebraic one. 

\paragraph{Geometrical analogue}

The geometric formulation of Hamiltonian systems deals with the symplectic form or the Poisson bracket on a symplectic or Poisson manifold. As a consequence, if one want to follow the geometric approach one has to be sure that the previous notions keep sense. This is done for example by C\'ami and Ortega \cite{cami} where they obtain a definition similar to the Bismut one but extends the notion up to cover Poisson Hamiltonian systems.
 
\paragraph{What about the It\^o case ?}

The It\^o version of stochastic Hamiltonian system as defined by Bismut in the Stratonovich case can be obtain using the {\it conversion formula} (\ref{strato-ito}) between Stratonovich and It\^o stochastic differential equations. The expression can be given using the {\bf Hamiltonian Swartz operator} as described by Ortega-Cam\'i in (\cite{cami}, Section 2.1, Proposition 2.8). Also, as remarked by Ortega-Cam\'i in (\cite{cami},Remark 2.5) many authors prefer to use the It\^o formalism as the definition of the It\^o integral is not {\bf anticipative}, that is, it does not assume any knowledge of the system in future times. This is of course the classical problem of {\bf causality} in Physics which is not respected by a Stratonovich modeling. Ortega and Cam\'i (\cite{cami},Remark 2.5) says that their definition also share this property as the Stratonovich definition can be translated in the It\^o framework. This is indeed the case, but the corresponding equation then possess a very different structure which can be difficultly called a Hamiltonian system. Indeed, the deterministic part is changed due to the noise and as a consequence, we obtain a new term which was not present in the initial perturbative scheme. The interpretation is then more difficult from the point of view of modeling.\\

It must be noted that in the It\^o framework, different strategies to construct a definition of a stochastic Hamiltonian system are explored for example in \cite{cresson} and \cite{zambrini} where the analogue of derivatives for stochastic process defined by E. Nelson \cite{nelson} are used.

\subsection{Stochastisation of Linearly dependent Hamiltonian systems}
\label{stochlinearhamiltonian}

We have discussed in Section \ref{stochastisationparameter}, a natural extension of linearly parameter dependent systems in the Stratonovich and It\^o setting. In this Section, we apply this approach to the Hamiltonian case.

\subsubsection{External stochastisation of Hamiltonian systems}

We consider a Hamiltonian function $H(z;p):=H_p (z)$ where $z$ belongs to a Poisson manifold $(P,\{ .,.\} )$ and the parameters $p\in \R^m$. We assume moreover that $H$ is linear with respect to $p$. We consider the Hamiltonian equation 
\begin{equation}
\dot{z} =X_{H_p ,\{ .,.\}} (z) .
\end{equation}
The external stochastisation procedure lead to a family of stochastic Hamiltonian systems of the form 
\begin{equation}
\label{externalhamiltonian}
dZ_t = X_{H_p ,\{ .,.\}} (Z_t) dt + X_{H_{\sigma (Z_t ,t) \star dB_t} ,\{ .,.\}} (Z_t) ,
\end{equation}
where $\sigma (Z_t ,t)$ is an $m \times k$ matrix and $dB_t$ is a $k$ dimensional Brownian motion, the symbol $\star$ is zero in the It\^o case and $\circ$ in the Stratonovich case.\\

If we begin with an Hamiltonian system in canonical coordinates where $z\in \R^{2n}$, we obtain in the symplectic case :
\begin{equation}
dZ_t = \mathbb{J} \nabla_z {H} (Z_t ;p) dt + \mathbb{J} \nabla_z {H} (Z_t ; \sigma( Z_t ,t) \star dB_t ) .
\end{equation}
We remark that if we assume that $B_t$ is a $n$-dimensional Brownian motion and $p$ is a $k$ dimensional parameter space, we obtain taking $\sigma (z,t)$ a constant $m\times m$ matrix of the form 
\begin{equation}
\sigma =\left ( 
\begin{array}{c}
\mathbb{I}
\end{array}
\right ) 
\end{equation}
a stochastic differential equation of the form
\begin{equation}
dZ_t = \mathbb{J} \nabla_z {H} (Z_t ;p) dt + \mathbb{J} \nabla_z {H} (Z_t ;  \star dB_t ) .
\end{equation}

This formula looks like the definition of stochastic Hamiltonian system by J-M. Bismut in \cite{bismut}. The two construction coincide when we restrict our attention to a specific form of Hamiltonian systems:

\begin{lemma}
Assume that $\mathbb{H} :\R^{2n} \times \R^m$ is a linear Hamiltonian system of the form 
\begin{equation}
\mathbb{H}(z;p)=\di\sum_{i=1}^m p_i H_i (z) ,
\end{equation}
where $H_i :\R^{2n} \rightarrow \R$, $i=1,\dots ,m$. Let us consider the Hamiltonian equation associated to $\mathbb{H}$ in canonical coordinates. An external perturbation of this equation under a noise with a drift given by a $m\times m$ matrix of the form 
\begin{equation}
\sigma =\left ( 
\begin{array}{c}
\mathbb{I}
\end{array}
\right ) 
\end{equation}
where $\mathbb{I}$ is the $m\times m$ identity matrix, leads to a stochastic differential equation of the form
\begin{equation}
dZ_t = \mathbb{J} \nabla_z {\mathbb{H}} (Z_t ;p) dt + \di\sum_{i=1}^m \mathbb{J} \nabla_z {H_i} (Z_t ) \star dB_t^i .
\end{equation}
\end{lemma}

We recognize the Bismut's definition of a stochastic Hamiltonian system in the Stratonovich case.

\subsubsection{Stochastic perturbation of isochronous Hamiltonian systems}

In this Section, we consider a family of linear Hamiltonian of the form 
\begin{equation}
H_i (I,\theta )=\omega_i I ,\ \ (I,\theta )\in \R\times S^1 ,
\end{equation}
and $\omega_i \in \R$ is a frequency. \\

We consider the $n$-degree of freedom {\bf isochronous Hamiltonian} defined by
\begin{equation}
H(I,\theta )=\omega_1 I_1 +\dots +\omega_n I_n ,\ \ (I,\theta )\in \R\times S^1 .
\end{equation}
The Hamiltonian system associated to $H$ is a set of $n$-oscillators
\begin{equation}
\left \{ 
\begin{array}{lll}
\dot{I}_i & = & 0,\\
\dot{\theta}_i & = & \omega_i .
\end{array}
\right .
\ \ i=1,\dots ,n.
\end{equation}
This equation is an example of {\bf integrable Hamiltonian system} as the solutions can be computed by quadrature. The terminology of {\bf isochronous} system comes from the fact that the frequency of the oscillators is the same for all the initial data.\\

We now make an external stochastic perturbation of the frequency $\omega =(\omega_1 ,\dots ,\omega_n )$. We then obtain the following stochastic differential equation system
\begin{equation}
\label{externalharmonic}
\left \{ 
\begin{array}{lll}
dI_i & = & 0\\
d\theta_i & = & \omega_i +\sigma (I,\theta ,t) \star dB_t 
\end{array}
\right .
\end{equation}
In the Stratonovich setting, this stochastic differential equation is not a stochastic Hamiltonian system unless the vector $(0 ,\sigma (I,\theta ,t)$ can be written as $\mathbb{J} \nabla_{(I,\theta)} K$ for a given scalar function $K(I,\theta )$. As a consequence, $K$ must be a function of $I$ only, i.e. that the diffusion term is also provided by an integrable Hamiltonian system. We then have:

\begin{lemma}
The external stochastic perturbation of the Harmonic Hamiltonian (\ref{externalharmonic}) lead to a stochastic Hamiltonian in the sense of Bismut if and only if the diffusion term is given by integrable Hamiltonian system. 
\end{lemma}

The Hamiltonian assumption on the diffusion can then induce an important change in the dynamics: we begin with an isochronous systems but adding a small Hamiltonian noise can for example produce a {\bf Harmonic oscillator} for which the frequency is dependent of the initial torus on which the dynamics takes place. 

\subsubsection{Stochastic Poisson structure of the Larmor equation}

The previous computations can be used for the Larmor equation which possesses a Poisson Hamiltonian structure (see Section \ref{poissonlarmor}, Lemma \ref{poissonlarmorlemma}) attached to the linear Hamiltonian $H(z;b)=H_b (z)=z.b$ for a vector $b\in \R^3$ fixed and $z\in \R^3$ and the Poisson bracket $\{ . , .\}_-$. The Hamiltonian vector field associated to $H$ and $\{ . ,.\}_-$ is linear in $b$ and given by $X_H (z) = z\wedge b$ (see Lemma \ref{hamlarmorpoisson}). Following equation (\ref{externalhamiltonian}), under a stochastic external perturbation the Larmor equation leads to 
\begin{equation}
dZ_t = X_{H_b ,\{ .,.\}} (Z_t) dt + X_{H_{\sigma (Z_t ,t) \star dB_t} ,\{ .,.\}} (Z_t) .
\end{equation}
Using our previous result of Section \ref{larmor}, we know that the sphere $S^2$ is always preserved (see Lemma \ref{larmorinvariance}) but the set of equilibrium points is preserved if and only the diffusion term is of the form $b\gamma (z) dB_t$ where $B_t$ is one dimensional (see Lemma \ref{larmorinvariance}). Taking these result into account, we define the following stochastic version of the Larmor equation :
\begin{equation}
\label{larmorstocalter}
dZ_t = X_{H_b ,\{ .,.\}} (Z_t) dt + X_{H_{b\gamma (Z_t ,t) \star dB_t} ,\{ .,.\}} (Z_t) ,
\end{equation}
where $\gamma (Z_t ,t)$ is a scalar function and $B_t$ is a one dimensional Brownian motion. Due to the linearity of $X_{H_b}$ with respect to $b$, equation (\ref{larmorstocalter}) can be rewritten as
\begin{equation}
\label{larmorstocalter2}
dZ_t = X_{H_b ,\{ .,.\}} (Z_t) dt + X_{H_b ,\{ .,.\}} (Z_t) \gamma (Z_t ,t) \star dB_t .
\end{equation}
We then recover a stochastic Hamiltonian system if the function $\gamma (Z_t ,t)$ is a constant $\gamma \in \R$:
\begin{equation}
\label{larmorstocalter3}
dZ_t = X_{H_b ,\{ .,.\}} (Z_t) dt + \gamma X_{H_b ,\{ .,.\}} (Z_t) \star dB_t .
\end{equation}

We can resume the previous discussion in the following Lemma:

\begin{lemma}
For all $\gamma \in \R^*$, the stochastic Larmor equation 
\begin{equation}
\label{larmorstocalter4}
dZ_t = (Z_t \wedge b) (dt + \gamma \star dB_t ) ,
\end{equation}
possesses the following properties:
\begin{itemize}
\item The sphere $S^2$ is invariant under the flow of equation (\ref{larmorstocalter4}).
\item The equilibrium points $\pm b/\parallel b\parallel$ of the Larmor equation persist in the stochastic case.
\item The Poisson Hamiltonian structure is preserved with the Poisson Hamiltonian $\mathbb{H} =\{ H_b ,\gamma H_b \}$.
\end{itemize}
\end{lemma}

The stochastic Larmor equation (\ref{larmorstocalter4}) preserves all the properties of the initial equation. Of course, relaxing the set of constraints, in particular concerning the persistence of the equilibrium points, leads to a lager class of model.

\subsection{Hamiltonian RODE}

The previous discussion shows that depending on the framework, it is not so easy to find a satisfying definition of a stochastic Hamiltonian system. In this Section, we consider an Hamiltonian system on a Poisson manifold $(P,\{ .,. \} )$ associated to a parameter dependent Hamiltonian function $H (z;p)=H_p (z)$ and given by
\begin{equation}
\dot{z} =X_{H_p , \{ .,.\}} (z) .
\end{equation}
Let $\eta_t$ be a given stochastic process. A $\eta_t$-RODE stochastic version of a Poisson Hamiltonian system is then of the form
\begin{equation}
\label{rodehamilton}
\dot{z} =X_{H_{\eta_t (\omega )} , \{ .,.\}} (z) ,
\end{equation}
for all realization $\omega \in \Omega$. 

\subsubsection{Properties of RODE Hamiltonian systems}

All the properties of Poisson Hamiltonian systems are preserved under any $\eta_t$-RODE stochastisation. In particular, we have :

\begin{lemma}
\label{bracketrode}
A function $F\in \mathcal{F} (P)$ is a first integral of the stochastic $\eta_t$-RODE Hamiltonian equation (\ref{rodehamilton}) if and only if $\{ F,H_{\eta_t (\omega )} \} =0$ for all $\omega \in \Omega$ and $t\in \R$. 
\end{lemma}

\begin{proof}
Let $\phi_t (\omega )$ be the flow generated by equation (\ref{rodehamilton}). We have 
\begin{equation}
\left .
\begin{array}{lll}
\di\frac{d}{dt} [F(\phi_t (\omega )] & = & DF (\phi_t (\omega )) .\di\frac{d\phi_t (\omega )}{dt} ,\\
    & = & DF (\phi_t (\omega )) .X_{H_{\eta_t}, \{ .,.\}} (\phi_t (\omega )) ,\\
     & = & \{ F,H_{\eta_t} \} 
\end{array}
\right .
\end{equation}
A function $F$ is a first integral of equation (\ref{rodehamilton}) if and only if $\di\frac{d}{dt} [F(\phi_t (\omega )] =0$ which is equivalent to $\{ F,H_{\eta_t} \} =0$. This concludes the proof.
\end{proof}

A useful result is obtain for the Casimir functions of the Poisson structure. Indeed, we obtain:

\begin{lemma}
\label{casimirintegral}
A Casimir function for the Poisson structure $\{ .,.\}$ is a first integral is a first integral of any stochastic $\eta_t$-RODE Hamiltonian system (\ref{rodehamilton}).
\end{lemma}

\begin{proof}
It $F$ is a Casimir function then for any function $G$ we have $\{ F,G \}=0$. By Lemma \ref{bracketrode}, a function $F$ is a first integral for equation (\ref{rodehamilton}) if $\{ F, H_{\eta_t (\omega )} \}=0$ which is always the case when $F$ is a Casimir function.
\end{proof} 

\subsubsection{Poisson Hamiltonian structure of a stochastic RODE Larmor equation}

We refer to Sections \ref{larmor1} and \ref{larmor2} for the definition of a RODE Larmor equation. Using the previous Section, we can easily obtain an alternative proof of Lemma \ref{invariancelarmor}:\\

By Lemma \ref{casimir}, the function $F(z)=\di\frac{\parallel z \parallel^2}{2}$ is a Casimir function of the Poisson manifold $(\R^3 ,\{ .,.\}_{\pm} )$. Using Lemma \ref{casimirintegral}, we deduce that $F$ is also a first integral of any RODE Larmor equation. As a consequence, any RODE Larmor equation preserves the sphere $S^2$.

\subsection{Double bracket dissipation and Hamiltonian structure}

In this Section, we discuss some particular dissipative systems for which the dissipation term can be described using a double bracket (see \cite{ratiu}). A typical example is given by the Landau-Lifshitz equation in ferromagnetism. We then study how this structure behaves under a stochastic perturbation. 

\subsubsection{Double bracket dissipation}

We follow \cite{ratiu} for this Section. We consider systems of the form
\begin{equation}
\dot{F} =\{ F,H\} -\{\{F,H\}\} ,
\end{equation}
where $H$ is the total energy of the system, $\{ F ,H\}$ is a skew symmetric bracket which is a Poisson bracket in the usual sense and where $\{ \{ F,H\} \}$ is a symmetric bracket.\\

The main property of these systems is that energy is dissipated but angular momentum is not. \\

Using the Poisson tensor denoted by $\Lambda$, given for all $z\in P$ by $\Lambda(z) : T_z^* P \rightarrow T_z P$ and $\Lambda (dH)=X_H$, i.e. 
\begin{equation}
\langle dF ,\Lambda (dH )\rangle =\{ F ,H\} ,
\end{equation}
we can write the previous double bracket as (see \cite{ratiu},$\S$.7):
\begin{equation}
\{\{ F,H\}\} =-\langle F,\Lambda \alpha \Lambda dH \rangle =\alpha (X_F ,X_H ) ,
\end{equation}
where $\alpha_z$ is a map from $T_z S \rightarrow T_z^* S$, where $S$ is the symplectic leaf through $z$, induced by a Riemannian metric defined on each leaf of $P$.\\

Using these maps, one can write the previous equation as 
\begin{equation}
\di\frac{dz}{dt} = \Lambda_z dH (z) +\Lambda_z \alpha_z \Lambda_z dH (z) ,
\end{equation}
where $\Lambda_z \alpha_z \Lambda_z dH (z)$ is the dissipation term.

\subsubsection{The Lie-Poisson instability theorem}

An important feature of these dissipative terms is that they do not destroy the equilibrium. Precisely, we have (see \cite{ratiu},Proposition 7.1):

\begin{proposition}
\label{bracketequilibrium}
If $z_e$ is an equilibrium for a Hamiltonian system with Hamiltonian $H$ on a Poisson manifold, then it is also an equilibrium for the system with added dissipative term of the form $\Lambda \alpha \Lambda dH$ as above, or double bracket form on the dual of a Lie algebra.
\end{proposition}

We refer to \cite{ratiu} for a proof.

\subsubsection{Example: The Landau-Lifshitz equation}
\label{landaulifshitzdoublebracket}

In this Section, we prove that the Landau-Lifshitz equation given by 
\begin{equation}
\label{llequation}
\di\frac{dz}{dt} =z\wedge b -\alpha z \wedge (z\wedge b) ,
\end{equation}
for $z\in \R^3$ can be written as a dissipative system with as a perturbation of a Poisson Hamiltonian system under a double bracket dissipation term, i.e. that there exist a Poisson bracket $\{ .,.\}$, a double bracket $\{\{ .,.\}\}$ and a Hamiltonian $H(z)$ such that equation (\ref{llequation}) is given by 
\begin{equation}
\di\frac{dz}{dt} =X_{H,\{ .,.\}} (z) +X_{H,\{\{ .,.\}\}} (z), 
\end{equation}
where $X_{H,\{ .,.\}} (z)$ and $X_{H,\{\{ .,.\}\}} (z)$ satisfy 
\begin{equation}
X_{H,\{ .,.\}} [F]=\{ F ,H\},\ \ \mbox{and}\ \ \ X_{H,\{\{ .,.\}\}} [F]=\{\{ F,H\}\},
\end{equation}
respectively. \\

For $\alpha=0$ the Landau-Lifshitz equation reduces to the Larmor equation whose Poisson structure was already studied in Section \ref{poissonlarmor}. In particular, the Poisson bracket is given by (see equation \ref{bracketlarmor})
\begin{equation}
\{ F ,G \}_{\pm} (z)=\pm z.\left ( \nabla F \wedge \nabla G \right ) ,
\end{equation}
and the Hamiltonian is $H(z)=z.b$. \\

Following \cite{ratiu} and \cite{liu}, the double bracket for the Landau-Lifshitz equation is given by:
\begin{equation}
\label{doublebracketll}
\{\{ F,G\}\} (z)=\alpha [z\wedge \nabla F(z)]\cdot [z\wedge \nabla G (z) ] .
\end{equation}
We derive easily the dissipation term induced by this double bracket:

\begin{lemma}
The dissipation term induced by the double bracket (\ref{doublebracketll}) with the Hamiltonian $H(z)=z.b$ is given by $-\alpha z\wedge (z\wedge b)$.
\end{lemma}

\begin{proof}
We have 
\begin{equation}
\left .
\begin{array}{lll}
\{ \{ F, H\}\} & = & \alpha [z\wedge \nabla F (z)] \cdot [z\wedge b ] ,\\
 & = & -\alpha [z\wedge b] \cdot [\nabla F (z) \wedge b] ,\\
  & = & \alpha \nabla F (z) \cdot ((z\wedge b)\wedge z ) ,\\
  & = & \nabla F (z) \left [ -\alpha z \wedge (z\wedge b) \right ] .
\end{array}
\right .
\end{equation}
The vector field associated to this double bracket is defined by
\begin{equation}
X_{H,\{\{ .,.\}\}} [F]=\{\{ F,H\}\} =\nabla F (z)\cdot X_{H,\{\{ .,.\}\}} (z) ,
\end{equation}
from which we deduces that 
\begin{equation}
X_{H,\{\{ .,.\}\}} (z) =-\alpha z \wedge (z\wedge b) .
\end{equation}
This concludes the proof.
\end{proof}

As a corollary, we obtain the following result:

\begin{corollary}
The Landau-Lifshitz equation is a double bracket dissipative Hamiltonian system over $(\R^3 ,\{ .,.\}_-)$ and the double bracket $\{\{ .,.\}\}$ defined by equation (\ref{doublebracketll}).
\end{corollary}

An interesting consequence of this structure is that, using Proposition \ref{bracketequilibrium}, we obtain directly the following Lemma:

\begin{lemma}
The Landau-Lifshitz equation (\ref{doublebracketll}) possesses the same equilibrium points as the Larmor equation, i.e. $\pm b/\parallel b\parallel$.
\end{lemma}
 
\subsection{Stochastic perturbation of Double bracket dissipative Hamiltonian systems}

\subsubsection{Stochastic double bracket dissipative Hamiltonian systems}

Following our previous discussion on stochastic Hamiltonian systems as defined by J-M. Bismut in \cite{bismut}, we define stochastic double bracket Hamiltonian systems as follows:

\begin{definition}
Let $\{ .,.\}$ and $\{\{ .,.\}\}$ be a Poisson bracket and double bracket on a Poisson manifold $P$. Let $\mathbf{H}=\{ H_i \}_{i=0,\dots ,r}$ be a finite family of real (or complex) valued functions defined on $P$. A stochastic double bracket dissipative Hamiltonian system of Hamiltonian $\mathbf{H}$ is an equation of the form
\begin{equation}
\di\frac{dz}{dt} =\left [ X_{H_0 ,\{ .,.\}} (z) +X_{H_0 ,\{\{ .,.\}\}} (z)\right ] dt +
\left [ \di\sum_{i=1}^r  X_{H_i ,\{ .,.\}} (z) +X_{H_i ,\{\{ .,.\}\}} (z) dW_{i,t} \right ],
\end{equation}
where $X_{H_j ,\{ .,.\}} (z)$ and $X_{H_j,\{\{ .,.\}\}} (z)$ satisfy 
\begin{equation}
X_{H,\{ .,.\}} [F]=\{ F ,H\},\ \ \mbox{and}\ \ \ X_{H,\{\{ .,.\}\}} [F]=\{\{ F,H\}\},
\end{equation}
respectively for $j=0,\dots ,r$ and $(W_1 ,\dots ,W_r)$ is a $r$-dimensional Brownian motion. 
\end{definition}

As for stochastic Hamiltonian systems, one can recover the main properties of double bracket dissipative Hamiltonian systems in the stochastic case. In the following, we focus on a particular class which will be of importance in the last Part concerning the Landau-Lifshitz equation. 

\subsubsection{Stochastic perturbation of linearly dependent double bracket Hamiltonian systems}

We consider a Hamiltonian function $H(z;p):=H_p (z)$ where $z$ belongs to a Poisson manifold $(P,\{ .,.\} )$ and the parameters $p\in \R^m$. Let $\{\{ .,.\}\}$ be a double bracket on $P$. We assume moreover that $H$ is linear with respect to $p$. We consider the double bracket dissipative  Hamiltonian equation 
\begin{equation}
\dot{z} =X_{H_p ,\{ .,.\}} (z)  )+ X_{H_p ,\{\{ .,.\}\}} (z).
\end{equation}
The external stochastisation procedure lead to a family of stochastic Hamiltonian systems of the form 
\begin{equation}
\label{externalhamiltonian2}
dZ_t = \left [ X_{H_p ,\{ .,.\}} (Z_t) + X_{H_p ,\{\{ .,.\}\}} (Z_t ) \right ] dt + \left [ X_{H_{\sigma (Z_t ,t) \star dB_t} ,\{ .,.\}} (Z_t) +X_{H_{\sigma (Z_t ,t) \star dB_t} ,\{\{ .,.\}\}} (z) \right ] ,
\end{equation}
where $\sigma (Z_t ,t)$ is an $m \times k$ matrix and $dB_t$ is a $k$ dimensional Brownian motion, the symbol $\star$ is zero in the It\^o case and $\circ$ in the Stratonovich case.\\

This formula looks like the definition of stochastic double bracket Hamiltonian system that we have just defined. The two construction coincide when we restrict our attention to a specific form of Hamiltonian systems:

\begin{lemma}
\label{lemmadoublebracketstochastic}
Assume that $\mathbb{H} :\R^{2n} \times \R^m$ is a linear Hamiltonian system of the form 
\begin{equation}
\mathbb{H}(z;p)=\di\sum_{i=1}^m p_i H_i (z) ,
\end{equation}
where $H_i :\R^{2n} \rightarrow \R$, $i=1,\dots ,m$. Let us consider the Hamiltonian equation associated to $\mathbb{H}$ in canonical coordinates. An external perturbation of this equation under a noise with a drift given by a constant $m\times m$ diagonal matrix with eigenvalues $(\sigma_1 ,\dots ,\sigma_m )$, leads to a stochastic differential equation of the form
\begin{equation}
dZ_t =\left [ X_{H_p ,\{ .,.\}} (Z_t) + X_{H_p ,\{\{ .,.\}\}} (Z_t ) \right ] dt + \di\sum_{i=1}^r \left [ X_{\sigma_i H_i ,\{ .,.\}} (Z_t) +X_{\sigma_i H_i ,\{\{ .,.\}\}} (z) \right ] dB_{i,t} ,
\end{equation}
i.e. a stochastic double bracket dissipative Hamiltonian system.
\end{lemma}

We recognize the definition of a stochastic double bracket dissipative Hamiltonian system in the Stratonovich case.

\subsubsection{A stochastic double bracket Landau-Lifshitz equation}

In Section \ref{landau}, a stochastic version of the classical Landau-Lifshitz equation is discussed. A possibility is given by Etore and al. in \cite{Etore}, where an external perturbation is used on the effective magnetic field and leads to the following equation
\begin{equation}
\label{stochasticll}
d \mu_t = [ -\mu_t \wedge b- \alpha \mu_{t}\wedge ( \mu_{t}\wedge b )] dt+ \varepsilon [ -\mu_{t}\wedge \star dW_{t}- \alpha \mu_t \wedge \mu_t \wedge \star dW_t ] ,
\tag{ELL}
\end{equation}
where the symbol $\star$ is omitted in the It\^o case and is replaced by $\circ$ in the Stratonovich case.\\

As already proved in Section \ref{landaulifshitzdoublebracket}, the deterministic part is associated to the Hamiltonian function $H(z)=z.b$ and the Poisson bracket $\{ .,.\}$ and double bracket $\{\{ .,.\}\}$ on $\R^3$ defined in equation (\ref{bracketlarmor}) and (\ref{doublebracketll}) respectively.\\

The previous double bracket Hamiltonian structure is preserved in the stochastic case:

\begin{lemma}
The stochastic Landau-Lifshitz equation (\ref{stochasticll}) is a stochastic double bracket dissipative Hamiltonian system with Hamiltonian $H_0 (z)=z.b$ and $H_i (z)=z.e_i$ for $i=1,2,3$ where $\{ e_i\}_{i=1,2,3}$ is the canonical base of $\R^3$.
\end{lemma}

\begin{proof}
We have only to write the stochastic part as a sum of Poisson or double bracket Hamiltonian vector field. If $H(z)=z.b$, then 
$X_{H,\{\{ .,.\}\}} =-\alpha z\wedge (z\wedge b)$. As a consequence, the term associated to the Hamiltonian $H_i (z)=z.e_i$ is given by $-\alpha z\wedge (z\wedge e_i)$. Writing 
\begin{equation}
\left .
\begin{array}{lll}
\di\sum_{i=1}^3 X_{H_i,\{\{ .,.\}\}} dB_{i,t} & = & -\di\sum_{i=1}^3 \alpha z\wedge (z\wedge e_i)dB_{i,t} ,\\
 & = & -\alpha z\wedge (z\wedge dB_t ) ,
\end{array}
\right .
\end{equation}
we deduce that equation (\ref{stochasticll}) can be written as
\begin{equation}
d\mu_t = X_{H_0 ,\{ .,.\}} +X_{H_0 ,\{\{ .,.\}\}} +\di\sum_{i=1}^3 \left [ X_{H_i ,\{ .,.\}} +X_{H_i,\{\{ .,.\}\}} \right ] dB_{i,t} .
\end{equation}
This concludes the proof.
\end{proof}

\section{Application: stochastic Landau-Lifshitz equations}
\label{landau}

In this part, we derive several models of a stochastic Landau-Lifshitz equation following different approaches to the modeling of the stochastic behavior of the effective magnetic field. We first remind the construction of the classical Landau-Lifshitz equation and then its main properties. We then review classical stochastic approach used by different authors and the difficulties associated with these models. We finally propose a new model based on the RODE  stochastisation process and under some selection principles. 

\subsection{The Landau-Lifshitz equation}

The Landau-Lifshitz equation is a generalization of the classical Larmor equation. The Larmor equation is {\bf conservative}. However, {\bf dissipative processes} take place within dynamic {\bf magnetization processes}. The microscopic nature of this dissipation is still not clear and is currently the focus of considerable research \cite{arrott,bertotti2}. The approach followed by Landau and Lifshitz consists of introducing dissipation in a phenomenological way. They introduce an additional torque term that pushes magnetization in the direction of the effective field. The Landau-Lifshitz equation becomes
\begin{equation}
\label{LLg2}
\frac{d\mu}{dt}=-\mu\wedge b -\alpha\mu \wedge(\mu \wedge b),
\tag{LLg}
\end{equation} 
where $\mu \in \R^3$ is the single magnetic moment, $\wedge$ is the vector cross product in $\mathbb{R}^3$, $b$ is the effective field and $\alpha >0$ is the damping effects.\\

As for the Larmor equation, this equation possess many particular properties which can be used to derive a stochastic analogue. We review some of them in the next Section.
 
\subsection{Properties of the Landau-Lifshitz equation}

In this Section, we give a {\bf self-contained} presentation of some classical features of the LL equation concerning equilibrium points, their stability and the asymptotic behavior of the solutions. In particular, we prove that the LL equation is an example of a {\bf gradient-like} dynamical systems on the sphere $S^2$. Classical properties of these systems on compact metric spaces explain the qualitative behavior of the solutions. Readers which are familiar with the LL equation can switch this Section.  

\subsubsection{Invariance}

The following result is fundamental is all the stochastic generalization of the LL equation.

\begin{lemma}
\label{invarianceLL}
Let $\mu (0) \in S^2$, then the solution $\mu_t$ satisfies for all $t\in \R$, $\parallel\mu_t \parallel =1$, i.e. the sphere $S^2$ is invariant under the flow of the LL equation.
\end{lemma}

We give the proof for the convenience of the reader.

\begin{proof}
Let $\mu_t$ be a solution of the LL equation. We have 
\begin{equation}
\left .
\begin{array}{lll}
\di\frac{d}{dt} \left [ \mu_t . \mu_t \right ] & = & 2 \mu_t .\di \frac{d\mu_t }{dt} ,\\
 & = & \mu_t . \left [ -\mu_t \wedge b -\alpha\mu_t \wedge(\mu_t \wedge b) \right ] .
\end{array}
\right .
\end{equation}
By definition of the wedge product, the vectors $\mu_t\wedge b$ and  $\alpha\mu_t \wedge(\mu_t \wedge b)$ are orthogonal to $\mu_t$ so that 
\begin{equation} 
\di\frac{d}{dt} \left [ \mu_t . \mu_t \right ] =0 .
\end{equation}
As a consequence, using the fact that $\mu_0 \in S^2$, we deduce that 
\begin{equation}
\parallel \mu_t \parallel =\parallel \mu_0 \parallel =1 ,
\end{equation}
which concludes the proof.
\end{proof}

As a consequence, a solution starting on the sphere $S^2$ will remains always on it. The sphere being a two dimensional compact manifold, we can use classical result to deduce the asymptotic behavior of the solutions. But first, let us compute the equilibrium points.

\subsubsection{Equilibrium points}

The equilibrium points of the LL equation are easily obtained.

\begin{lemma}
The LL equation possesses as equilibrium points $b/\parallel b\parallel$ and $-b/\parallel b\parallel$. 
\end{lemma}

We give the proof for the convenience of the reader.

\begin{proof}
An equilibrium point $\mu \in \R^3$ satisfies 
\begin{equation}
-\mu\wedge b -\alpha\mu \wedge(\mu \wedge b)=0,
\end{equation}
which gives 
\begin{equation}
-\mu\wedge b =\alpha\mu \wedge(\mu \wedge b) .
\end{equation}
The vector $\mu$ must be orthogonal to $\mu \wedge b$ and at the same time equal to $-\mu \wedge b$ up to a factor $\alpha >0$. As $\mu_0 \in S^2$, we have $\mu\not= 0$ and the only solution is 
\begin{equation}
\mu \wedge b =0 .
\end{equation}
We then obtain $\mu =\lambda b$, with $\lambda \in \R$. By Lemma \ref{invarianceLL}, we must have $\mu \in S^2$ so that $\lambda = \pm 1/\parallel b\parallel$. This concludes the proof.
\end{proof}

We see that the equilibrium point of the LL equation coincide with those of the Larmor equation.\\

The stability of the previous equilibrium point can be easily studied using the  Lyapunov theory.

\subsubsection{Lyapunov function and stability of the equilibrium points}

A common way to study the stability of equilibrium points is to find a  Lyapunov function. The LL equation possesses a global  Lyapunov function and as a consequence is an example of a {\bf gradient-like dynamical system}.

\begin{lemma}
The function $V(\mu )=-\mu . b$ is a strict  Lyapunov function for the LL equation.
\end{lemma}

\begin{proof}
This is a simple computation. We have 
\begin{equation}
\left .
\begin{array}{lll}
\di\frac{d}{dt} \left [ V(\mu_t ) \right ] & = & - \di\frac{d\mu_t}{dt} . b ,\\
 & = & \left [ \mu_t \wedge b + \alpha\mu_t \wedge(\mu_t \wedge b) \right ] . b ,\\
 & = &  - \alpha \parallel \mu_t \wedge b \parallel^2 , 
\end{array}
\right .
\end{equation}
using the property of the {\it mixed product} of vectors $u,v,w$ defined by $[u,v,w]=(u\wedge v).w$ that $[u,v,w]=-[u,w,v]$.\\

As the damping coefficient $\alpha >0$, we deduce that $\di\frac{d}{dt} \left [ V(\mu_t ) \right ]<0$ except at the equilibrium points. This concludes the proof.
\end{proof}

\subsubsection{Asymptotic behavior}

The classical {\bf Poincar\'e-Bendixon Theorem} can be formulated for the sphere $S^2$ (see \cite{palis},p.18): ;

\begin{theorem}[Poincar\'e-Bendixon] Let $X$ be a vector field on the sphere $S^2$ with a finite number of singularities. Take $p\in S^2$ and let $\omega (p)$ be the $\omega$-limit set of $p$. Then one of the following possibilities holds:
\begin{itemize}
\item $\omega (p)$ is a singularity;
\item $\omega (p)$ is a closed orbit;
\item $\omega (p)$ consists of singularities $p_1$,$\dots$,$p_n$ and regular orbits such that if $\gamma \subset \omega (p)$ then $\alpha (\gamma )=p_i$ and $\omega (\gamma )=p_j$.
\end{itemize}
\end{theorem}

As we have only two equilibrium points, we have the following asymptotic behavior.

\begin{lemma}
For all $\mu_0 \in S^2 \setminus \left \{ -b/\parallel b\parallel \right \}$, the associated solution $\mu_t$ satisfies 
\begin{equation}
\di\lim_{t\rightarrow +\infty} \mu_t =\di\frac{b}{\parallel b\parallel} .
\end{equation}
\end{lemma}

\begin{proof}
By the Poincar\'e-Bendixon Theorem on $S^2$, the $\omega$ limit set of a given trajectory not starting in an equilibrium point is either a periodic orbit or an equilibrium point. However, as $V$ is a strictly decreasing function on the solutions, the LL system can not have periodic solutions. As a consequence, the $\omega$ limit set of a given trajectory is an equilibrium point. This point must be a minimum for $V$ so that $\mu=b/\parallel b\parallel$ is the only solution.
\end{proof}

This result shows that classical results on the LL equation are related to general properties of dynamical systems on the sphere $S^2$ and in particular to its gradient-like structure.

\subsubsection{Simulations}

The previous results can be illustrated with some numerical simulations. We take $b=(0,0,1)$, $\alpha=1$ and we consider a set of initial conditions starting in a neighborhood of the South pole. We clearly see that the solutions are going to the North pole as predicted.  

\begin{center}
\begin{figure}[!h]
\centerline{%
    \begin{tabular}{cc}
        \includegraphics[width=0.4\textwidth]{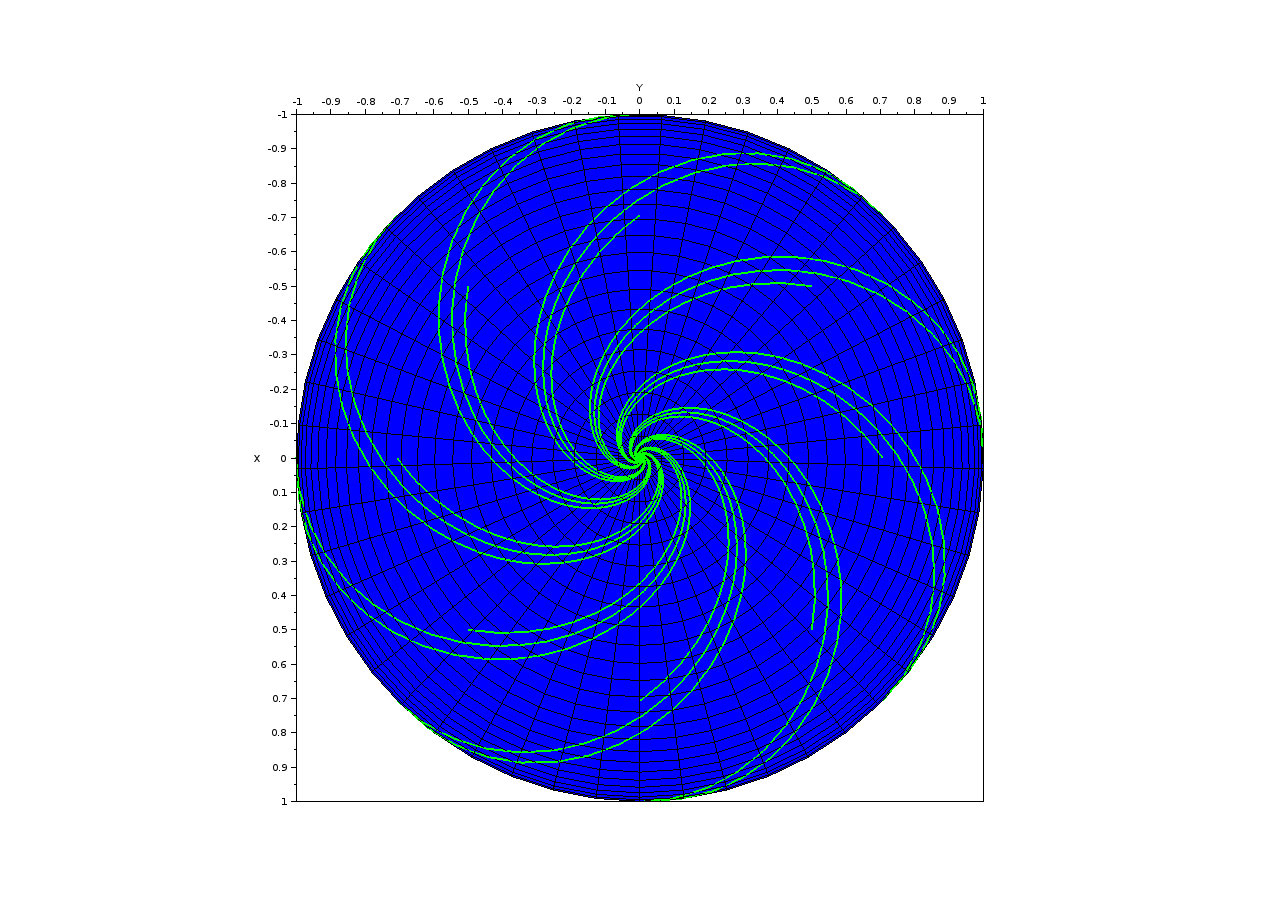} & \includegraphics[width=0.4\textwidth]{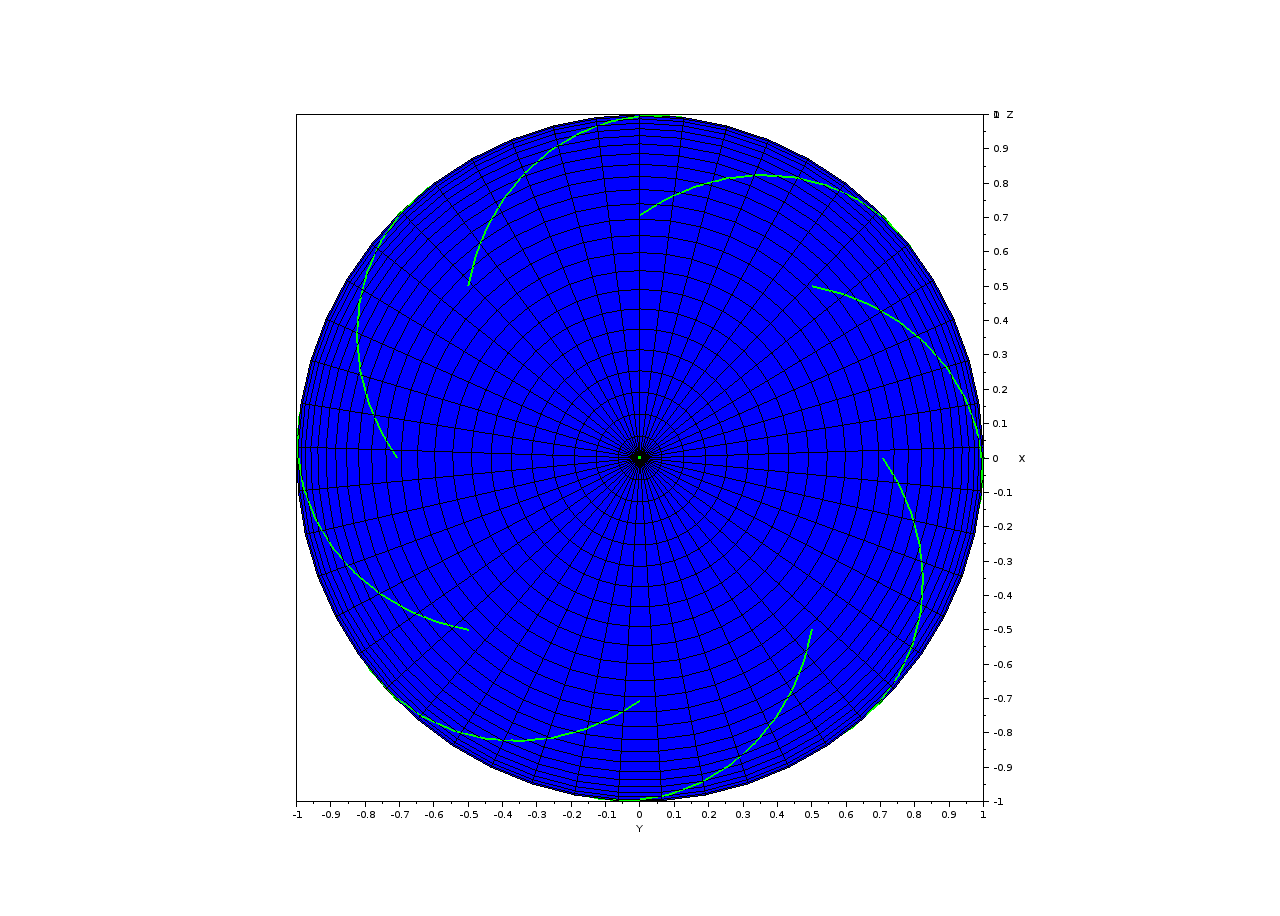}\\
        \footnotesize{(a) Dynamics near the North pole.} & \footnotesize{(b) Dynamics near the South pole}\\
    \end{tabular}}
    \caption{The deterministic Landau-Lifshitz equation}
\end{figure}
\label{ex1fig1}
\end{center}

A global view is given by :

\begin{figure}[ht!]
	\centering
	\includegraphics[width=0.5\textwidth]{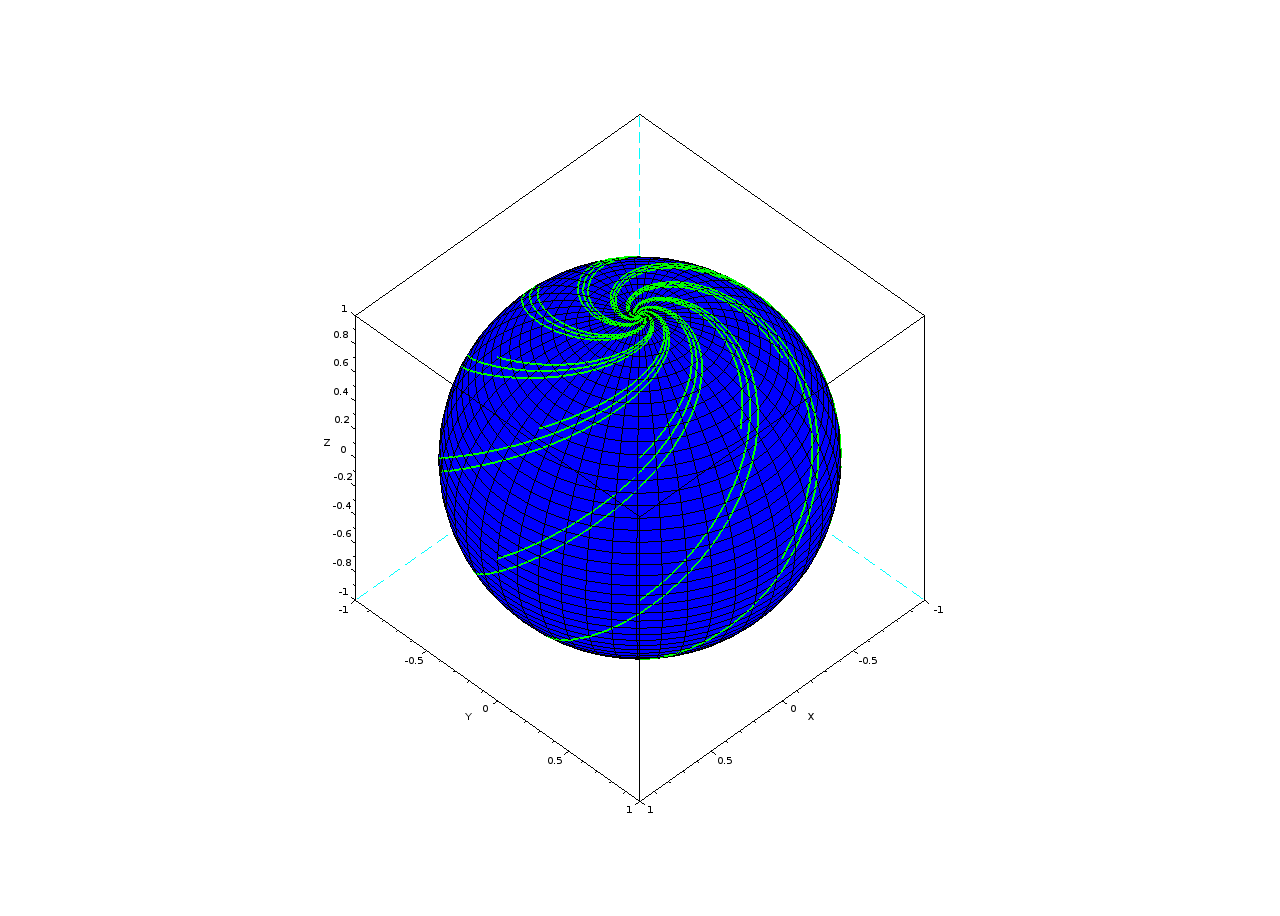}
	\caption{Global view of the dynamics}
\end{figure}

\subsection{Toward a stochastic Landau-Lifshitz equation}

In this Section, we discuss the usual way of deriving a stochastic analogue of the Landau-Lifshitz equation by considering an external perturbation of the effective magnetic field. We focus on the Stratonovich and the It\^o interpretation and we explain the strategy used in Etore and al. \cite{Etore} to bypass the obstruction that the It\^o version does not preserve the sphere $S^2$ using the invariantization method. 

\subsubsection{Classical approach to the stochastic Landau-Lifshitz equation}  

The main approach to deal with the stochastic behavior of the effective magnetic field is to use the external stochastisation procedure as exposed in Section \ref{stochastisationparameter} to obtain a stochastic differential equation in It\^o or Stratonovich sense of the following form:

\begin{equation}
\label{ELL}
d \mu_t = [ -\mu_t \wedge b- \alpha \mu_{t}\wedge ( \mu_{t}\wedge b )] dt+ \varepsilon [ -\mu_{t}\wedge \star dW_{t}- \alpha \mu_t \wedge \mu_t \wedge \star dW_t ] ,
\tag{ELL}
\end{equation}
where the symbol $\star$ is omitted in the It\^o case and is replaced by $\circ$ in the Stratonovich case.\\

Most of the authors use the Stratonovich formalism in order to give a sense to the previous equation. The main reason is that in this case, one has directly using Theorem \ref{sis}:

\begin{lemma}
The sphere $S^2$ is invariant under the flow of the Stratonovich version of equation (\ref{ELL}).
\end{lemma}

\begin{proof}
We use Corollary \ref{spheres}. We have directly that 
\begin{equation}
\mu \cdot [ -\mu_t \wedge b- \alpha \mu_{t}\wedge ( \mu_{t}\wedge b )] =0 ,
\end{equation}
and 
\begin{equation}
\mu \cdot [ -\mu_{t}\wedge dW_{t}- \alpha \mu_t \wedge \mu_t \wedge dW_t ] =0 ,
\end{equation}
which concludes the proof.
\end{proof}

However, as pointed out by Etore and al. in \cite{Etore}, the Stratonovich version of the Landau-Lifshitz equation leads to several difficulties: 
\begin{itemize}
\item The stability of the equilibrium points of the deterministic LL equation is lost. 
\end{itemize}

The previous point has motivated the work \cite{Etore} in which the authors discuss the It\^o case. However, the It\^o approach lead to other difficulties. Details are given in the next Section.

\subsubsection{Etore and al. model for a stochastic Landau-Lifshitz equation}  

The It\^o version of the stochastic Landau-Lifchitz equation possesses many drawback and the main one follows directly from the invariance criterion that we derive in Corollary \ref{spherei}. Indeed, we have:

\begin{lemma}
The sphere $S^2$ is not invariant under the flow of the It\^o version of equation (\ref{ELL}).
\end{lemma}

\begin{proof}
By Corollary \ref{spherei}, the diffusion term must be zero on the sphere $S^2$ and all $t\in \R$. The It\^o version of the Landau-Lifshitz equation has the form
\begin{equation}
d \mu_t = [ -\mu_t \wedge b- \alpha \mu_{t}\wedge ( \mu_{t}\wedge b )] dt+ \varepsilon  \sigma(t,\mu_t )dW_t,
\end{equation}
where $\alpha \not=0$ and
\begin{equation}
\label{sigmaetore}
\sigma(t,x)= -x \wedge . -\alpha x\wedge(x\wedge .)= \begin{pmatrix}
\alpha(x_3^2+x_2^2) &  x_3-\alpha x_1x_2 &-x_2 -\alpha x_3x_1\\
-x_3-\alpha x_1x_2 &\alpha (x_3^2+x_1^2 )&x_1 -\alpha x_3x_2 \\
x_2-\alpha x_1 x_3 & -x_1-\alpha x_3x_2 & \alpha(x_2^2+x_1^2) 
\end{pmatrix}
\end{equation}
and $W_t$ is a $3$-dimentional Brownian motion. The condition on $\sigma$ on the diagonal terms implies that $\alpha=0$, i.e. that we can not have a dissipative term and we recover the Larmor equation in contradiction with our assumption that $\alpha \not= 0$. This concludes the proof.
\end{proof}

\begin{center}
\begin{figure}[!ht]
\centerline{\includegraphics[width=0.6\textwidth]{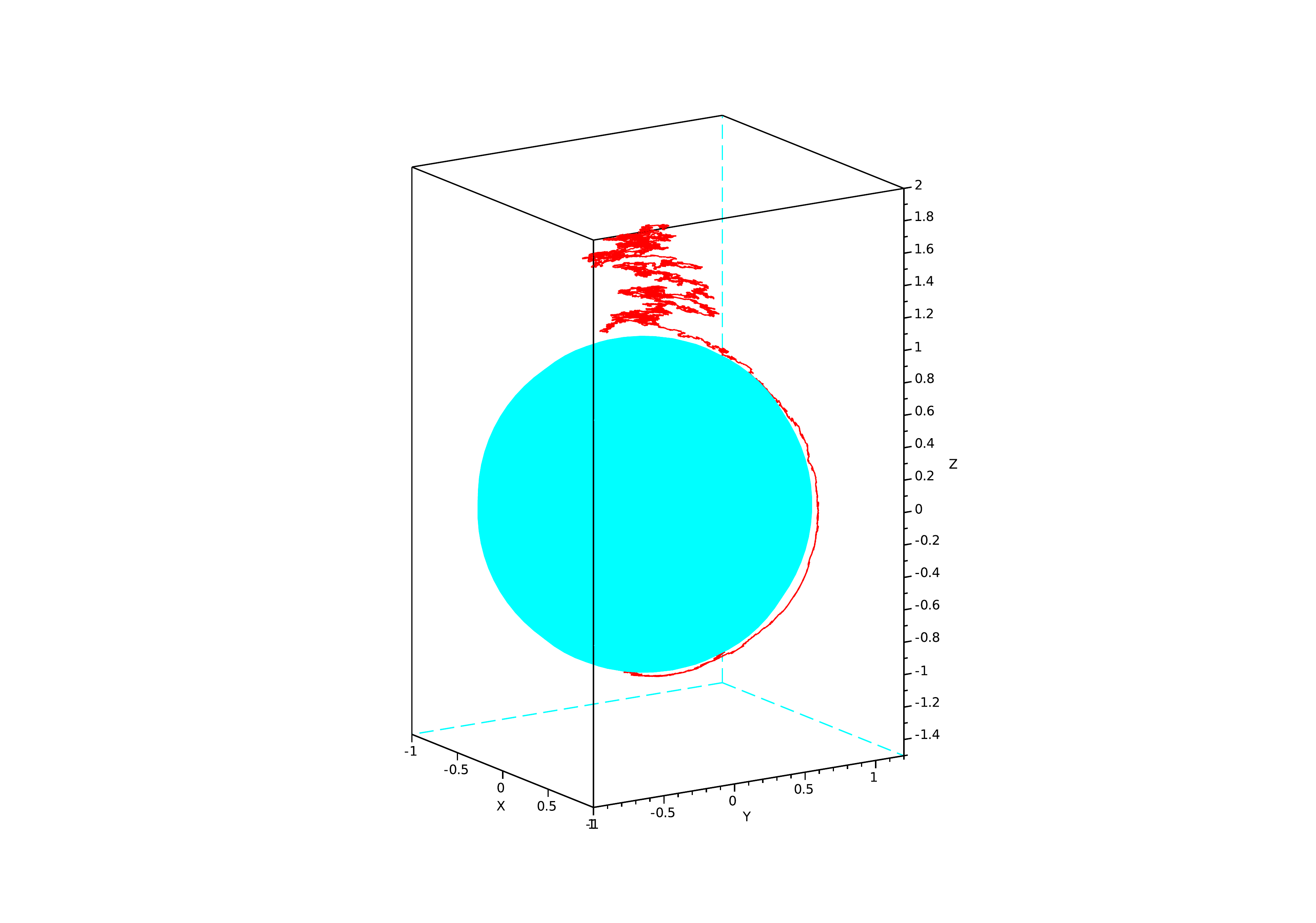}} 
\caption{A solution of the Ito version of equation (\ref{ELL}).}
\end{figure}
\end{center}

Moreover, the  Lyapunov function $V(\mu )=\mu .b$ and the first integral $I (\mu )=\parallel \mu \parallel^2$ does not behaves correctly contrary to the Stratonovich version of the equation.

\begin{center}
\begin{figure}[!ht]
\centerline{%
    \begin{tabular}{cc}
        \includegraphics[width=0.4\textwidth]{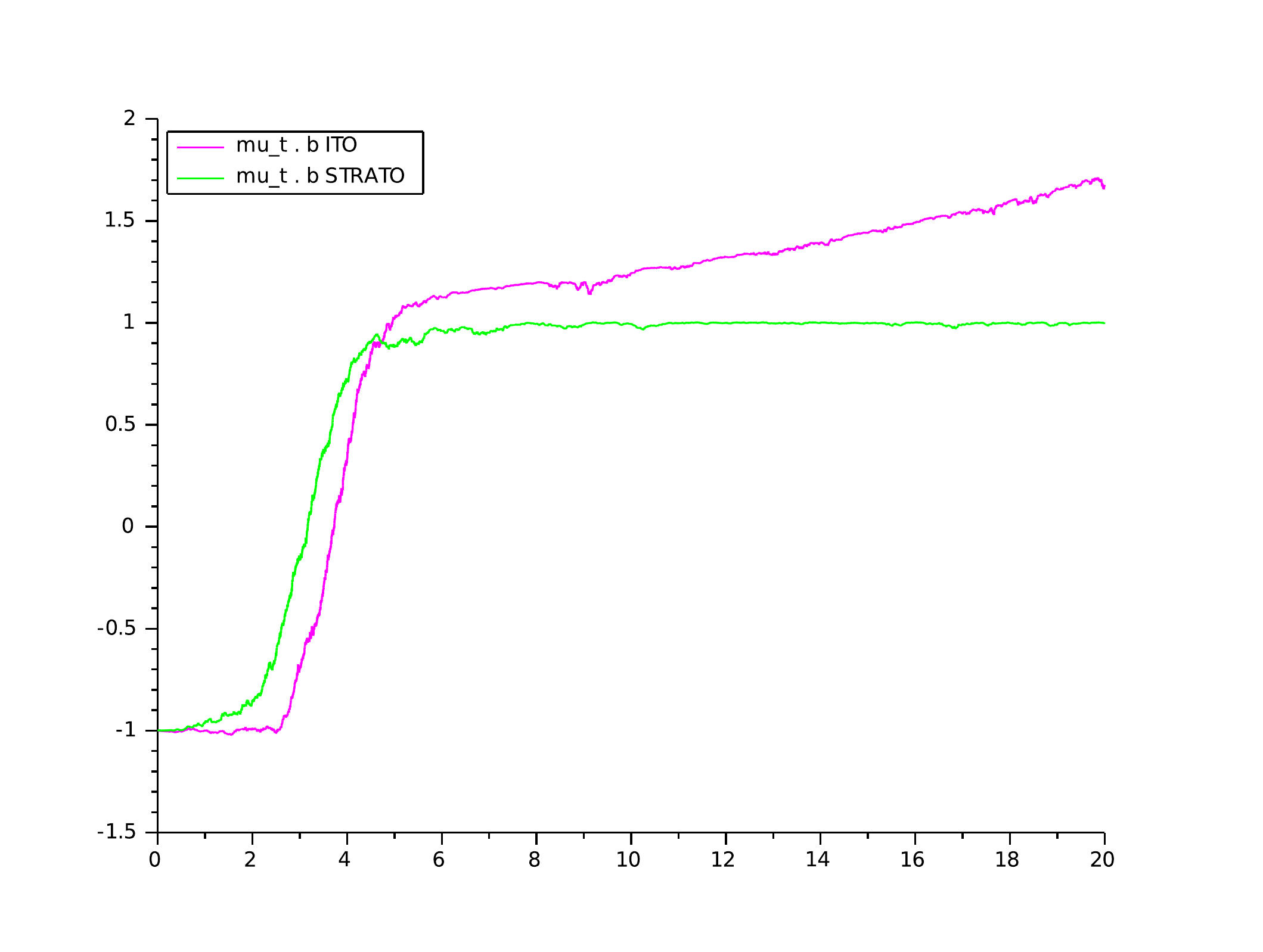} & \includegraphics[width=0.4\textwidth]{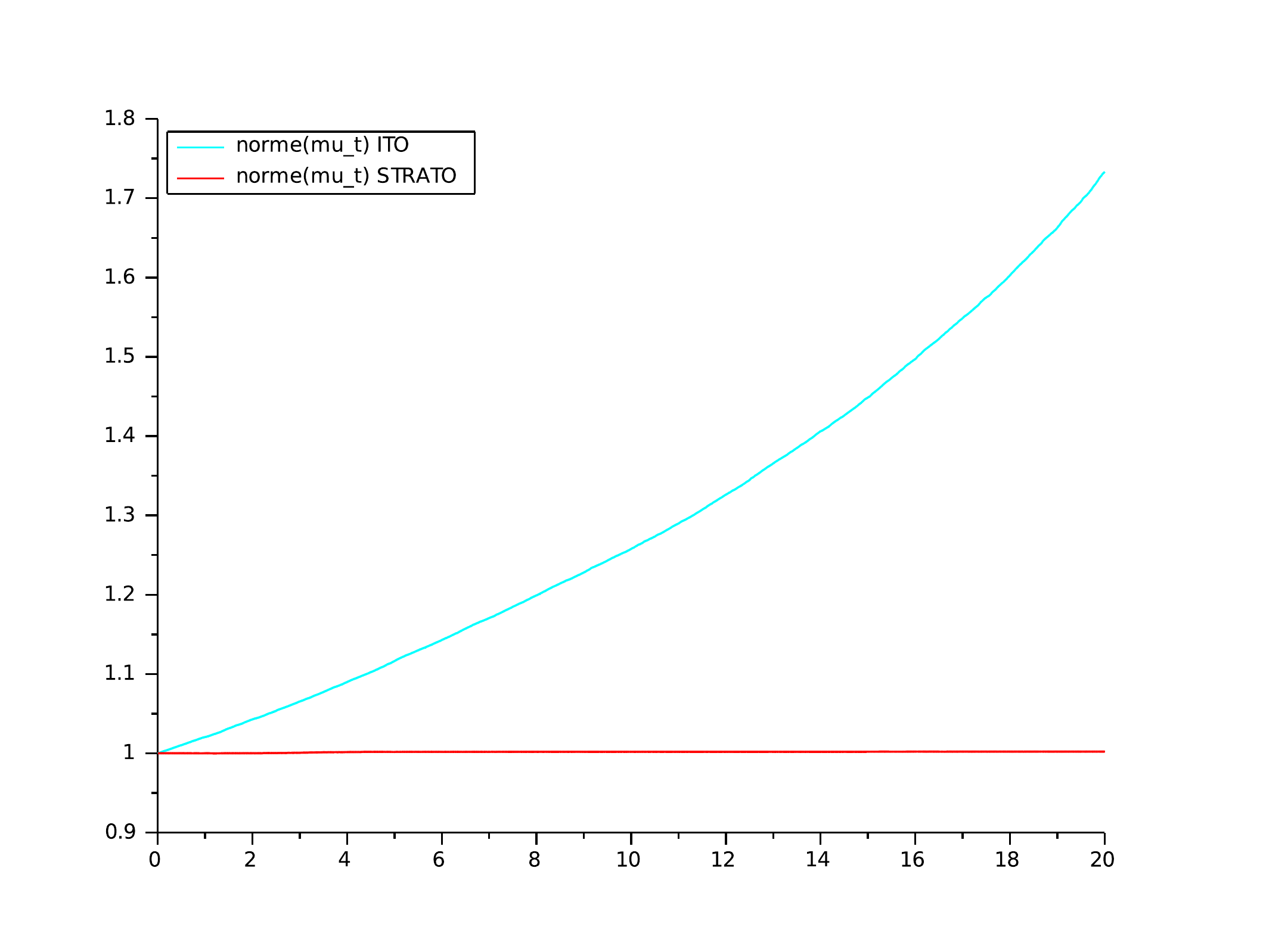}\\
        \footnotesize{(a) Behavior of the  Lyapunov function $V(\mu )=\mu.b$.} & \footnotesize{(b) Behavior of $\parallel \mu_t \parallel^2$}\\
    \end{tabular}}
    \caption{Comparaison between the Ito and Stratonovich version of equation (\ref{ELL}).}
\end{figure}
\end{center}

The previous result excludes the use of the It\^o formalism for a direct stochastic generalization of the Landau-Lifshitz equation. However, we can use the {\it invariantization method} exposed in \cite{cy} in order to obtain an It\^o model, related to the previous one, but which satisfies the invariance of the sphere $S^2$.

\paragraph{Invariance of $S^2$ and equilibrium points}

Using the invariantization method exposed in \cite{cy}, we recover the {\bf Etore and al. version of the Stochastic Landau-Lifshitz equation} defined by:

\begin{equation}
\label{etoreequation2}
dx_t=\left[ -\frac{1}{2}\frac{ 2\epsilon^2 (\alpha^2 +1)}{2\epsilon^2 (\alpha^2 +1)t+1}x_t+\frac{1}{\sqrt{2\epsilon^2 (\alpha^2 +1)t+1}}f(t,x_t)\right]  dt+\frac{1}{\sqrt{2\epsilon^2 (\alpha^2 +1)t+1} }\sigma(t,x_t)dW_t ,
\end{equation}
where $\sigma$ is defined by equation (\ref{sigmaetore}).\\

However, such a model does not preserves the equilibrium points of the initial system. Indeed, we have:

\begin{lemma}
The points $\mu =\pm b /\parallel b\parallel$ are not equilibrium points of equation (\ref{etoreequation2}).
\end{lemma}

\begin{proof}
For $\mu=\pm b/\parallel b\parallel$, we have for all $v\in \R^3$ that  
\begin{equation}
\sigma(t,\pm b/\parallel b\parallel )[v]= -\di\frac{b}{\parallel b\parallel} \wedge v -\pm \alpha \di\frac{b}{\parallel b\parallel} \wedge( \pm\di\frac{b}{\parallel b\parallel} \wedge v) .
\end{equation}
Denoting by $w$ the vector $\di\frac{b}{\parallel b\parallel} \wedge v$ the previous equality can be written as 
\begin{equation}
\sigma(t,\pm b/\parallel b\parallel )[v]=-w -\alpha \di\frac{b}{\parallel b\parallel} \wedge w ,
\end{equation}
and is zero if 
\begin{equation}
w= -\alpha \di\frac{b}{\parallel b\parallel} \wedge w .
\end{equation}
As $\alpha \di\frac{b}{\parallel b\parallel} \wedge w$ is a vector orthogonal to $w$, this equality is satisfied if and only if 
$w=0$, i.e. $\di\frac{b}{\parallel b\parallel} \wedge v =0$ for all $v\in \R^3$. This is possible if and only if $b=0$ which is in contradiction with the assumption that the effective field is non zero.
\end{proof}

It must be noted that the previous problem can be easily solved by modifying a little bit the modeling of the stochastic behavior of the effective magnetic field. Indeed, let us consider instead of $dW_t$ the following vector 
\begin{equation}
b dW_t ,
\end{equation}
where $W_t$ is a one dimensional Brownian motion. This assumptions is equivalent to say that we consider only stochastic behavior in the direction of the initial field $b$. This is of course very particular, but in this case the new model preserves the equilibrium points of the initial system:

\begin{lemma}
Let us consider the modified Etore and al. stochastic Landau-Lifshitz equation defined by 
\begin{equation}
\label{modifiedetoreequation}
dx_t=\left[ -\frac{1}{2}\frac{ 2\epsilon^2 (\alpha^2 +1)}{2\epsilon^2 (\alpha^2 +1)t+1}x_t+\frac{1}{\sqrt{2\epsilon^2 (\alpha^2 +1)t+1}}f(t,x_t)\right]  dt+\frac{1}{\sqrt{2\epsilon^2 (\alpha^2 +1)t+1} }\sigma(t,x_t) \left [ bdW_t \right ] ,
\end{equation}
where $\sigma$ is defined by equation (\ref{sigmaetore}) and $W_t$ is a one dimensional Brownian motion. This equation possesses as equilibrium points $\pm b/\parallel b\parallel$.
\end{lemma}

\begin{proof}
This follows easily from the previous proof only saying that $v$ is always collinear to $b$.
\end{proof}

\paragraph{Stability of the equilibrium points}

We consider the modified stochastic Landau-Lifshitz equation (\ref{modifiedetoreequation}). In this case, the equilibrium points are preserved so that one can study their stability. 

\begin{lemma}
The function $V(\mu )=-\mu .b$ satisfies 
\begin{equation}
\left .
\begin{array}{lll}
d\left [ V(\mu_t ) \right ] & = & \left[ -\di\frac{1}{2}\frac{ 2\epsilon^2 (\alpha^2 +1)}{2\epsilon^2 (\alpha^2 +1)t+1} V(\mu_t )- \alpha \di\frac{1}{\sqrt{2\epsilon^2 (\alpha^2 +1)t+1}}\parallel \mu_t \wedge b\parallel^2 \right]  dt\\
 & & -\alpha \di\frac{1}{\sqrt{2\epsilon^2 (\alpha^2 +1)t+1} }\parallel \mu_t \wedge b\parallel^2 dW_t .
\end{array}
\right .
\end{equation}
\end{lemma}

As a consequence, we obtain the following estimate:

\begin{lemma}
Let $L$ be the infinitesimal generator associated to the stochastic equation (\ref{modifiedetoreequation}). We have 
\begin{equation}
\left .
\begin{array}{lll}
L[V](\mu_t ) & = & \left[ -\di\frac{1}{2}\frac{ 2\epsilon^2 (\alpha^2 +1)}{2\epsilon^2 (\alpha^2 +1)t+1} V(\mu_t )+ 
\alpha \di\frac{1}{\sqrt{2\epsilon^2 (\alpha^2 +1)t+1}} V(\mu_t)^2 \right . \\
 & &\left . -\alpha \di\frac{1}{\sqrt{2\epsilon^2 (\alpha^2 +1)t+1}}\parallel \mu_t \parallel^2 \parallel  b\parallel^2 \right]  dt .
\end{array}
\right .
\end{equation}
\end{lemma}

\begin{proof}
We use the classical identity mixing the wedge product and the scalar product given for all vectors $u,v\in \R^3$ by 
\begin{equation}
\parallel u\wedge v\parallel^2 =\parallel u\parallel^2 \parallel v\parallel^2 -(u.v)^2 .
\end{equation}
\end{proof}

We then deduce from Theorem \ref{stability}:

\begin{theorem}
For all $\epsilon >0$ sufficiently small, the equilibrium point $b/\parallel b\parallel$ is asymptotically stable and the equilibrium point $-b/\parallel b\parallel$ is unstable.
\end{theorem}

We then recover the classical qualitative behavior of the deterministic system.

\subsection{A RODE approach to a stochastic Landau -Lifshitz equation}
\label{rode}

We propose a new stochastic model using the RODE stochastisation procedure. Depending on the properties of the underlying stochastic process used to construct the RODE model, the corresponding stochastic Landau-Lifshitz equation will preserve only the $S^2$ invariance or more of the classical properties of the classical Landau-Lifshitz equation.

\subsubsection{A RODE  Landau -Lifshitz model}

Assume that the external field on the deterministic LL equation is a stochastic process $\mathbf{b}_t$ defined on the probability space $(\Omega, \mathcal{F},\mathbb{P}),$ then we get the following model :
\begin{equation}
\label{RLLE}
d\mu =\bigl( -\mu\wedge \mathbf{b}(t,\omega)-\alpha\mu \wedge(\mu \wedge \mathbf{b}(t,\omega)) \bigr) dt,\  \text{for almost all }  \omega \in \Omega,
   \tag{RLLE}
\end{equation} 
where $\mathbf{b}_t$ is an $\mathbb{R}^{3}$-valued stochastic process with continuous sample paths, the initial value $\mu_{0}$ is a $\mathbb{R}^{3}$-valued random variable. \\

The solution of the stochastic equation \ref{RLLE} is a stochastic process $\mu(t)$, which is adapted to the filtration $\mathcal{F}$ exist and unique provided that the driving process $\mathbf{b}_t$ is $\mathcal{F}$-adapted and independent of the initial condition $\mu_{0},$ see \cite{kloeden2}.

\subsubsection{Invariance of the sphere $S^2$}

As for the RODE version of the Larmor equation, the invariance of $S^2$ is preserved without any condition on $b_t$. Precisely, we have:

\begin{theorem}
\label{rodei}
Assume that the stochastic process $\mathbf{b}_t$ is adapted to the filtration $\mathcal{F},$ and independent of the initial condition $\mu_{0}.$ The stochastic Landau-Lifshitz equation \ref{RLLE} preserve the invariance of the sphere $S^{2}$ if and only if, $\mathbf{b}_t$ is almost surely bounded, for all time $t\geq 0.$
\end{theorem}

\begin{proof}
This follows directly from Theorem \ref{persistencerode}.
\end{proof}

This result is crucial for applications and modeling. We can go further studying the persistence of equilibrium points.

\subsubsection{Equilibrium points}

We already know that the RODE version of the Larmor equation preserves the equilibrium points of the initial equation if and only if $b_t$ is $\langle b\rangle$-valued, i.e. of the form $b_t =b\eta_t$ with $\eta_t$ a real valued stochastic process. For the RODE version of the Landau-Lifshitz equation a similar result holds:

\begin{lemma}
The equilibrium points of the Landau-Lifshitz equation are preserved by a $b_t$-RODE stochastisation if and only if $b_t$ is of the form $b_0 \eta_t$ where $\eta_t$ is a real valued bounded stochastic process. The equilibrium points are then given by $\pm b_0 /\parallel b_0 \parallel$.
\end{lemma}

\begin{proof}
Equilibirum points of the stochastic equation (\ref{RLLE}) are solutions of
\begin{equation}
-\mu\wedge b_t -\alpha\mu \wedge(\mu \wedge b_t) =0 ,
\end{equation}
when $\alpha\not=0$. Properties of the mixed product implies that 
\begin{equation}
\mu \wedge\left [ b_t +\alpha \mu \wedge b_t \right ] =0 .
\end{equation} 
Then either $\mu$ is colinear to $b_t +\alpha\mu \wedge b_t$ or one these two vectors is zero. As $\mu \in S^2$, we can not have $\mu=0$ and the only possible solution is 
\begin{equation}
b_t +\alpha \mu \wedge b_t =0 .
\end{equation}
As $\mu\wedge b_t$ is orthogonal to $b_t$ the only possibility is $b_t =0$ which is not allowed. Hence, the vector $\mu$ is colinear to $b_t +\mu \wedge b_t$. As $\mu \wedge b_t$ is orthogonal to the linear space generated by $b_t$ and $\mu$, the vector $\mu$ can not be colinear to $b_t +\alpha \mu\wedge b_t$ unless $\mu \wedge b_t =0$, i.e. $\mu$ is colinear to $b_t$. \\

This equation as a solution if and only if $b_t$ keep a constant direction for all time, i.e. if $b_t = b_0 \eta_t$ where $\eta_t$ is a real valued bounded stochastic process. In this case, we obtain two equilibrium points given by $\pm b_0 /\parallel b_0 \parallel$ as in the deterministic case with $b=b_0$. This concludes the proof.
\end{proof}

\subsubsection{Stochastic stability}
   
The classical  Lyapunov function is $V(\mu )=-\mu .b$ and governs the global structure of the solutions. In the stochastic case, we recover exactly the same result as long as we take a stochastic effective field which preserves the properties of the initial system concerning equilibrium points.

\begin{lemma}
The function $V(\mu )=-\mu .b$ is a  Lyapunov function for the RODE Landau-Lifshitz equation.
\end{lemma}

\begin{proof}
We have $b_t =b\eta_t$ where $\eta_t$ is a positive bounded real valued stochastic process. As a consequence, we obtain
\begin{equation}
\left .
\begin{array}{lll}
\di\frac{d}{dt} \left [ V(\mu_t )\right ] & = & -\di\frac{d\mu}{dt} .b ,\\
 &= & [ \mu\wedge b_t +\alpha \mu \wedge (\mu \wedge b_t ) ].b .
\end{array}
\right .
\end{equation}
The first term is zero as $b_t$ is colinear to $b$ and the second one can be written using the properties of the mixed product as
\begin{equation}
\alpha [\mu \wedge (\mu \wedge b_t ) ].b = -\alpha \eta_t \parallel \mu_t \wedge b\parallel^2 .
\end{equation}
As a consequence, we obtain that 
\begin{equation}
\di\frac{d}{dt} \left [ V(\mu_t )\right ] =-\alpha \eta_t \parallel \mu_t \wedge b\parallel^2  <0,
\end{equation}
for almost all values of $\omega$ and $t\in \R$. This concludes the proof.
\end{proof}   

\subsubsection{An explicit example}

The previous properties about equilibrium points, stability and invariance are satisfied for all stochastic effective field of the form 
\begin{equation}
b_t =b .\eta_t ,
\end{equation}
where $b$ is a fixed vector of $\R^3$ and $\eta_t$ is a bounded real valued stochastic process. As an example, one can consider as a stochastic effective field the following process
\begin{equation}
b(t)=b. \exp[  \frac{B_{t}}{\sqrt{2t \log (\log t)}} )] .
\end{equation}
Indeed, by the {\it law of iterated logarithm} (see \cite[p.68]{oksendal2003stochastic}, Theorem 5.1.2) the process $\frac{B_{t}}{\sqrt{2t \log (\log t)}}$ is a bounded process and satisfies 
\begin{equation}
\lim\sup_{t\rightarrow \infty} \di \frac{B_{t}}{\sqrt{2t \log (\log t)}} =1\ \ a.s.
\end{equation}
As a consequence, our assumptions on $b_t$ are satisfied.

\subsubsection{Preserving the double bracket dissipative Hamiltonian structure}

The classical double bracket dissipative Hamiltonian structure of the Landau-Lifshitz equation is given in Part \ref{symppoisson}, Section \ref{landaulifshitzdoublebracket}. We have the following result:

\begin{lemma}
Any $b_t$-RODE stochastisation of the Landau-Lifshitz equation preserves the double bracket dissipative Hamiltonian structure.
\end{lemma}

\begin{proof}
This follows from the fact that all the computations made in Part \ref{symppoisson}, Section \ref{landaulifshitzdoublebracket} are valid if $b$ is replaced by $b_t$. 
\end{proof}

\section{Conclusion and perspectives}

\subsection{Numerical integrator for stochastic Landau-Lifshitz equations}

In \cite{aquino}, M. d'Aquino constructs {\bf geometric numerical integrator} for the Landau-Lifshitz-Gilbert equation. Using the double bracket structure of our stochastic RODE Landau-Lifshitz equation, it is tempting to develop the same kind of numerical integrator in the stochastic case.

\subsection{Hysteresis phenomenon}

A classical problem in ferromagnetism concerns the behavior of the system when the effective field is slowly varying between two opposite values, says $b$ and $-b$. When the effective field is fixed, the magnetic moment asymptotically stabilizes along the field. However, when the effective field is slowly varying, we wait for a competition between the stabilization rate of $\mu$ and the variation rate of $b$. In particular, as noted by Etore and al. \cite{Etore}, if the variation rate is sufficiently small, we expect that the magnetic moment takes different back and forth paths when the external field switches from $b$ to $-b$ and then from $-b$ to $b$. This characterizes the {\bf hysteresis behavior} of the system. It will be interesting to study this phenomenon for the class of proposed stochastic Landau-Lifshitz equations introduced in this paper. 

\begin{appendix}

\section{Reminder about stochastic invariance of submanifolds}
\label{sec-invariance}

In this section, we derive an {\bf invariance criterion} for a submanifold denoted by $M$ of codimension $1$ of $\mathbb{R}^{n}$ which correspond to the zero set of a given function $F:\mathbb{R}^{n}\rightarrow \mathbb{R} $ of class $C^{2},$ i.e.
\begin{equation}
\label{manifold}
M=\lbrace x\in \mathbb{R}^{n} \setminus F(x)=0 \rbrace ,
\end{equation}
under the flow of a stochastic differential equation in the It\^o or Stratonovich sense. This result is by itself not new and many general results are known in particular a {\bf stochastic Naguno-Brezis} Theorem as proved by Aubin-Da Prato in \cite{aubin} or A. Milian in \cite{milian}. However, most of these results are difficult to read for a non-specialist in the field of stochastic calculus. The main interest of the following computations are precisely that our criterion can be {\bf easily derived} using {\bf basic results} in stochastic calculus. 

\subsection{Geometric definition of invariance}

We consider an ordinary differential equation of the form 
\begin{equation} 
\label{DE1}
\left\{
\begin{array}{lll}
 \dot {x}_t & = & f (t,x_t ), \\
x(0) & = & x_0
\end{array}
\right.
\tag{ODE}
\end{equation}
where $f:\mathbb{R}^{+} \times \mathbb{R}^{n}\longrightarrow \mathbb{R}$ is a function of class $C^{1}$ and $x_0 \in \R^n$ is the initial condition.

\begin{definition}
A given submanifold $M \subset \mathbb{R}^{n}$ is said to be invariant under the flow of the differential equation \eqref{DE1} if for all $ x_{0}\in M ,$ the maximal solution $ x_{t}(x_{0})$ starting in $ x_{0} $ when $ t=0 $ satisfied $x_{t}(x_{0})\in M$ for all $ t\in \mathbb{R}^+.$
\end{definition}

We denote by $T_{x} M$ the tangent plane of $M$ at $x$, we can write the invariance condition as follows 
\begin{equation}
f(t,x)\in T_{x}M,\quad \text{ for all } (t,x)\in \mathbb{R}^+ \times M.
\end{equation}
As $M$ is of codimension 1, for all $x\in M $ we can define the normal vector $N(x) $ to the tangent hyperplane $ T_{x}M $ in $ x ,$ such that 
$$ T_{x} M =\lbrace y \in \mathbb{R}^{d}, y.N(x)=0\rbrace,$$   
then the invariance condition can be written as 
\begin{equation}
\label{IF}
N(x)\cdot f(t,x)=0, \quad \text{for all } (t,x)\in \mathbb{R}^+ \times M.
\end{equation}
When $M$ is of the form \eqref{manifold}, the normal vector to $M$ at $x$ is equal to $\nabla F(x).$ Then the invariance condition reads as 
\begin{equation}
\label{IFF}
\nabla F(x) \cdot f(t,x)=0,\quad \text{for all } (t,x)\in \mathbb{R}^+ \times M.
\tag{IF}
\end{equation}

In the stochastic case, the trajectories are continuous but nowhere differentiable. As a consequence the previous geometric condition can not be used. In the following we discuss two natural generalization of the notion of invariance in the stochastic setting.

\subsection{Strong stochastic invariance}

Let us consider a stochastic differential equation of the form \eqref{IE}. The stochastic character of the flow allows us to defined two natural notions of invariance.
 
\begin{definition}[Strong persistence]
A submanifold $M$ is invariant in the strong sense for the stochastic system \eqref{IE} if for every initial data $x_0\in M$ almost surely, the corresponding solution $x(t),$ satisfies 
\begin{equation}
\mathbb{P}\lbrace F\left( x(t)\right)  =0, t \in [ t_0, +\infty )\rbrace =1,
\end{equation}
i.e., the solution almost surely attains values within the submanifold $M.$
\end{definition}

\subsubsection{Strong invariance - the Kubo oscillator}
\label{kubo}

Let us consider the {\bf stochastic Kubo oscillator} model as defined by Milstein in \cite{mil}:
\begin{equation}
\label{Kubo oscillator}
\left\lbrace \begin{array}{lll}
dX_1&=&-a X_2dt-\sigma X_2\circ dW_t, \quad X_1(0)=1\\
dX_2&=&a X_1dt+\sigma X_1\circ dW_t, \quad X_2(0)=0
\end{array}
\right.
\end{equation}
where $X=\begin{pmatrix} X_1\\ X_2 \end{pmatrix}\in \R^{2}, a, \sigma \in\R$ and $W_t$ is a 1-dimensional Brownian motion.\\

These systems have specific properties. In particular, the function $H_0$, when it does not depend on time, is a {\bf first integral} corresponding to the {\bf energy} of the underlying system, i.e. that for all solution of the deterministic system $(x_{1,t} ,x_{2,t} ) \in \R^{2n}$, we have 
\begin{equation}
\di\frac{d}{dt} \left [H_0 (x_{1,t} ,x_{2,t} )\right ]=0.
\end{equation}
As a consequence, all the {\bf level surfaces} of the function $H_0$ are invariant under the flow of the deterministic system. In general, this property is destroyed by a stochastic perturbation.  However, as we will see, the Kubo oscillator preserves this particular features of the deterministic system.\\

In the Kubo oscillator case, the family of Hamiltonian is given by $\mathbf{H}=\{ H_0 (X)=X_1^2 +X_2^2 ,\ H_1=\sigma H_0 \}$, where $\sigma \in \R$. The previous result on the invariance of the level surfaces of the Hamiltonian $H_0$ is preserved under the stochastic perturbation. Precisely, we have:

\begin{lemma}
The function $H_0 (X)=X_1^2 +X_2^2$ is a strong first integral for the Kubo oscillator, i.e. 
\begin{equation}
d\left [ H_0 (X_{1,t} ,X_{2,t} )\right ] =0 ,
\end{equation}
over the solution of the system, meaning that the level surfaces of $H_0$ are strongly invariant under the flow of the Kubo oscillator.
\end{lemma}

\begin{proof}
We have 
\begin{equation}
\left .
\begin{array}{lll}
d [H_0 (X_{1,t} ,X_{2,t} )] & = &  2X_{1,t} dX_1 +2X_{2,t} dX_2 ,\\
 & = & -2aX_1 X_2 dt -2\sigma X_1 X_2 \circ dW_t +2aX_1 X_2 dt +2\sigma X_1 X_2 \circ dW_t ,\\
 & = & 0.
\end{array}
\right .
\end{equation}
This concludes the proof.
\end{proof}

As a consequence, any solution of the Kubo oscillator starting with an initial condition $(X_{1,0} ,X_{2,0} )$ will remain on the circle
$X_1^2 +X_2^2 =r_0^2$ with $r_0^2 =X_{1,0}^2 +X_{2,0}^2$.

\subsection{Explicit conditions for Invariance - It\^o case}

In this case, the strong persistence of the invariant submanifold $M$ is conditioned by the deterministic invariance condition (IF) imposed at the stochastic perturbation i.e.,  $ \sigma(t,x_{t})\in T_{x_{t}}M.$ Although, the form of $\sigma$ has not an influence in the weak persistence, where $ E \left[\int_{0}^t \nabla F(x_{s})\cdot \sigma(s,x_{s})dW_{s}\right]=0 $ for any stochastic perturbation.

 \begin{theorem}[It\^o 's strong invariance]
 \label{invarianceito}
 Let $M$ be a submanifold defined by a function $F$ invariant under the deterministic flow associated to \eqref{DE1}, i.e., 
 $$  \nabla F(x)\cdot f (t,x)=0, \text{ for all } x\in M, t\geq 0$$
The submanifold $ M $ is strongly invariant under the flow of the stochastic system \eqref{IE}, if and only if, 
 $$  \nabla F(x)\cdot \sigma (t,x)=0, \text{ for all } x\in M, t\geq 0$$
 and 
\begin{equation}
\label{sc}
\sum_{i,j}\frac{\partial^{2}F}{\partial x_{i}\partial x_{j}}(x_{t})\sum^{l}_{k=1}\sigma_{i,k}(t,x_{t})\sigma_{j,k}(t,x_{t}) =0.
\end{equation} 
\end{theorem}

\begin{proof}
The essential tool in this case is the It\^o formula that will help us to formulate the invariance condition. Indeed, a process $x_{t}$ leaves the submanifold $M$ invariant if and only if for all initial condition $x_{0}\in M$ a.s, the stochastic process associated to $x_{t}$ satisfies $F(x_{t})=0$ for all $t$ almost surely where it is defined.\\

The multidimensional It\^o formula reads as
$$d[F(x_{t})]=\nabla F(x_{t})dx_{t}+\frac{1}{2}\sum_{i,j}\frac{\partial^{2}F}{\partial x_{i}\partial x_{j}}(x_{t})dx_{i}(t)dx_{j}(t).$$
So we obtain
$$d[F(x_{t}]=\nabla F(x_{t})f(t,x_t)dt + \nabla F(x_{t})\sigma(t,x_{t})dW_t+\frac{1}{2}\sum_{i,j}\frac{\partial^{2}F}{\partial x_{i}\partial x_{j}}(x_{t})\sum^{l}_{k=1}\sigma_{i,k}(t,x_{t})\sigma_{j,k}(t,x_{t})dt.$$ 
The gradient of $F$  being always normal to the tangent space of $M$, we have $\nabla F(x_{t})\cdot f (t,x_{t})=0$ since the manifold $M$ is assumed to be invariant in the deterministic case. It remains
\begin{equation}
\label{ifF}
 d[F(x_{t}]=\nabla F(x_{t})\sigma(t,x_{t})dW_{t}+\frac{1}{2}\sum_{i,j}\frac{\partial^{2}F}{\partial x_{i}\partial x_{j}}(x_{t})\sum^{l}_{k=1}\sigma_{i,k}(t,x_{t})\sigma_{j,k}(t,x_{t})dt.
 \end{equation}
The only contribution to the stochastic part is given by $ \nabla F(x_{t})\sigma(t,x_{t}) $ and is equal to zero if and only if the perturbation $\sigma$ satisfies the invariance condition \eqref{IFF}. Then the previous expression reduces to:
\begin{equation}
\label{itoderiva}
d[F(x_{t}]= \frac{1}{2}\sum_{i,j}\frac{\partial^{2}F}{\partial x_{i}\partial x_{j}}(x_{t})\sum^{l}_{k=1}\sigma_{i,k}(t,x_{t})\sigma_{j,k}(t,x_{t})dt.
\end{equation}
that give us the third conditional.\\

If we assume that the stochastic perturbation take the simplest case, where $\sigma_{i,j}=\delta_{i}^{j},$ we get the condition 
$$ \sum_{i=1}^{n}\frac{\partial^{2}F}{\partial x_{i}^2}(x_{t})\sigma_{i,i}^{2}(t,x_{t})=0,\;\forall(t,x)\in \mathbb{R}^+\times M. $$
\end{proof}

The previous Theorem indicates that unless a very specific form for $\sigma$ and $F$, there is no hope to recover invariance of a given manifold using a direct stochastic perturbation of a deterministic equation in the It\^o case. 
  
\subsection{Explicit conditions for invariance - Stratonovich case}
 
We generalize the previous result in the Stratonovich case.
 
\begin{theorem}[Stratonovich stochastic invariance]
\label{sis}
The submanifold $M$ is invariant under the flow of the stochastic system in the Stratonovich sense if and only if
\begin{equation}
f(t,x)\in T_{x}M \text{ and }  \sigma(t,x)\in T_{x}M ,\text{ for all } (t,x)\in \mathbb{R}^+ \times M .
\end{equation}
\end{theorem}

\begin{proof}
As the Stratonovich stochastic calculus behaves as the usual differential calculus, we obtain 
$$d\left[ F(x_{t}\right] =\nabla F(x_{t})\cdot f(t,x_{t})dt+\nabla F(x_{t})\cdot \sigma(t,x_{t})\circ dW_{t}.$$
This quantity is zero if and only if $\nabla F(x_{t})\cdot f(t,x_{t})=0$ and $\nabla F(x_{t})\cdot \sigma(t,x_{t})=0$, which concludes the proof.
\end{proof}

The Stratonovich situation is then very close to the usual deterministic case.

\subsection{Stochastic invariance of the sphere $S^2$ in $\R^3$}

We can specialize this result in the case of the sphere which will be important to study the invariance property of Landau-Lifshitz equation.

\begin{corollary}
\label{spherei}
The sphere $ S$ is invariant under the flow of the stochastic system \eqref{IE} if and only if the stochastic perturbation is null on the sphere i.e.,
$$\sigma_{i,i}(t,x)=0, \quad i=1,...,n\text{ for all }\ t\in \mathbb{R}^+\text{ and }\ x\in S^{n-1} .$$
\end{corollary} 

\begin{proof}
The proof follows from the fact that $F(x)=\di\sum_{i=1}^n x_i^2$ so that condition \ref{sc} reduces to
\begin{equation}
\sum_{i=1}^{n}[\sigma_{i,i}(t,x_{t})]^{2}=0,\;\forall(t,x)\in \mathbb{R}^+\times S .
\end{equation}
This concludes the proof.
\end{proof}
  
An easy consequence of theorem\eqref{sis} is the following "explicit" criteria for the sphere in the Stratonovich case:

\begin{corollary} 
\label{spheres}
The sphere $S$ is invariant under the flow of the stochastic system \eqref{SE} if and only if  
$$x\cdot f(t,x)= x\cdot \sigma(t,x)=0, \;\forall(t,x)\in \mathbb{R}^+\times S.$$
\end{corollary}

\begin{proof}
This follows from the fact that a vector $N$ is normal to the tangent plane of $S^2$ at point $x$ if and only if $N\wedge x=0$. As a consequence, the vector $f(t,x)$ and $\sigma (t,x)$ belongs to the tangent manifold $T_x S^2$ at point $x$ if and only if 
$x \cdot f(t,x)=0$ and $x\cdot \sigma (t,x)=0$.
\end{proof}

\end{appendix}


\begin{thebibliography}{15}     
\bibitem{Arnold} L.Arnold, {\it Stochastic differential equations: theory and application}, Wiley Interscience, 1974.

\bibitem{AlE} Allen, E.J, {\it Modeling with It\^o Stochastic Differential Equations}, Math. Model. Theory Appl. 22, Springer, Dordrecht, 2007.

\bibitem{Al} Allen,  L.J.S, {\it An Introduction to Stochastic Processes with Applications to Biology}, 2$^{nd}$ edition, Chapman \& Hall, Boca Raton, 2011.

\bibitem{aleale} Allen, E.J., Allen, L.J.S., Archintega, A. and Greenwood O.E.,Construction of equivalent stochastic differential equation models, Stoch. Anal. Appl. 26(2): 274--297, 2008.

\bibitem{aquino} M. d'Aquino, Nonlinear Magnetization Dynamics in Thin-films and Nanoparticles, Ph. D, Universita degli studi di Napoli "Federico II", 2004.

\bibitem{arrott} A. Arrott, Plenary Lecture: Progress in Micromagnetics, Moscow International Symposium on Magnetism, Moscow (2002).

\bibitem{atkinson} D. Atkinson, D.A. Allwood, C.C. Faulkner, G. Xiong, M.D. Cooke, R.P. Cowburn, Magnetic domain wall dynamics in a permalloy nanowire, IEEE Trans. Magnetics 39(5), pp. 2663-2665, 2003.

\bibitem{aubin} J.P.Aubin, G.Da Prato, Stochastic viability and invariance, Scuola normale superiore, Pisa, 1990.

\bibitem{bertotti2} G. Bertotti, Tutorial Session Notes: Micromagnetics and Nonlinear Magnetization Dynamics, 10th Biennal Conference on Electromagnetic Fields Computation, Perugia (2002).

\bibitem{Bertotti} G.Bertotti, I.Mayergoys, C.Serpico, Non linear magnetization dynamics in nanosystems, Elsevier, 2009. 

\bibitem{bismut} J.M. Bismut, {\it M\'ecanique al\'eatoire}, volume 929, Lecture notes in mathematics, Springer-Verlag, 1981.

\bibitem{bloch} A. Bloch, P.S. Krishnaprasad, J.E. Marsden, T.S. Ratiu, The Euler-Poincar\'e equations and double bracket dissipation, Communication in Mathematical Physics Vol. 175, pp. 1-42, 1996.

\bibitem{Br} Braumann, C.A. It\^o versus Stratonovich calculus in random population growth, Math. Biosci. 206(1), 81--107, 2001.

\bibitem{brown1} W.-F. Brown, {\it magnetostatic principles in Ferromagnetism}, North-Holland, 1962.

\bibitem{brown2} W.-F. Brown, {\it Micromagnetics}, Interscience Publishers, 1963.

\bibitem{cami} J-A L\'azaro-Cami and J-P. Ortega, Stochastic Hamiltonian dynamical systems, Reports on Mathematical Physics, 61(1), 65-122, 2008.

\bibitem{conley} C. Conley, {\it Isolated Invariant Sets and the Morse Index}, Conference Board on Mathematical Sciences 38. AMS, Providence (1978).

\bibitem{cresson} J. Cresson, S. Darses, Stochastic embedding of dynamical systems, Journal of Math. Phys. 48, 072703 (2007).

\bibitem{cp} J. Cresson, F. Pierret et B. Puig : The Sharma-Parthasarathy stochastic two-body problem. Journal of Mathematical Physics, 56(3), 2015.

\bibitem{CrPuSo2} Cresson, J., Puig B., and Sonner, S. Stochastic models in biology and the invariance problem.
Discrete Contin. Dyn. Syst. Ser. B, 21(7): 2145--2168, 2016.

\bibitem{CrPuSo1} Cresson, J., Puig B., and Sonner, S. Validating stochastic models : invariance criteria for systems of stochastic differential equations and the selection of a stochastic Hodgkin-Huxley type model, Int. J. Biomath. Biostat., 2(1): 111--122, 2013.

\bibitem{CrSo} Cresson J., Sonner S., A note on a derivation method for SDE models: application in Biology and viability criteria, preprint, 28.p, 2016.

\bibitem{CrSz} Cresson J., A. Szafra\`{n}ska, Discrete and continuous fractional persistence problems -- the positivity property and applications, Communication in Nonlinear Science and Numerical Simulation, Volume 44, March 2017, Pages 424-448.

\bibitem{cy} Cresson J., Kheloufi Y., Nachi K., A stochastic invariantization method for It\^o stochastic perturbations of differential equations, Arxiv 1809.09363, 15 pages, 2018.

\bibitem{Cimrak} I.Cimr'k, A survey on the numerics and computations for Landau-Lifshitz equation of micromagnetism, Arch. Comput. Methods Eng. 2007.

\bibitem{ei1}    A. Einstein, On the movement of small particles suspended in a stationary liquid demanded by the molecular kinetic theory of heat. Ann. Phys. 17, 549-560, 1905. 

\bibitem{ei2} A. Einstein, On the theory of the Brownian movement. Ann. Phys. 19, 371-381, 1906. 
 
\bibitem{Etore} P. \'Etor\'e, S.Labb\'e , J. Lelong, Long time behaviour of a stochastic nanoparticle, J. Differential Equations 257 (2014), 2115-2135.
 
\bibitem{Federer} H. Federer, Geometric Measure Theory, Springer, 1969.

\bibitem{Fox}  R. F. Fox, Stochastic Versions of the Hodgkin-Huxley Equations, Biophysical Journal, 72, 2068-2074, 1997.

\bibitem{hirsch}  M. W. Hirsch, S. Smale, R. L. Devaney, {\it Differential Equations, Dynamical Systems, and an Introduction to Chaos}, 
3rd Edition, Academic Press, 2012.

\bibitem{karatzas} Karatzas I., Schreve S.E., {\it Brownian motion and stochastic calculus}, Second Edition, Graduate Texts in Mathematics 113, 1991, Springer Verlag, New York.

\bibitem{Khasminskii} Khasminskii.R, {\it Stochastic Stability of Differential Equations}, Second Edition, 2012, Springer.

\bibitem{Kloeden} P.Kloeden, A.Jentzen, Recent advances in the numerical approximation of stochastic partial differential equation, Lecture series, 2010.

\bibitem{kloeden2} X. Han, P.E. Kloeden, Random Ordinary Differential Equations and Their Numerical Solution, Probability Theory and Stochastic Modelling 85, Springer Nature Singapore Pte Ltd. 2017

\bibitem{Krasovskii}N. N. Krasovskii, j.L. Brenner, Stability of Motion: Applications of Lyapunov's Second Method to Differential Systems and Equations With Delay,Stanford university press, 1963.

\bibitem{Kozin} F. Kozin, Stability of the linear stochastic system.

\bibitem{Lakshmanan} M. Lakshmanan, The Fascinating World of Landau-Lifshitz-Gilbert Equation: An Overview, The Royal Society, 2011.

\bibitem{leimkuhler} B. Leimkuhler, C. Matthews, MV. Tretyakov, On the long-time integration of stochastic gradient systems.
Proc. R. Soc. A 470: 20140120, 2014.

\bibitem{liu} C-S. Liu, K-C. Chen, C-S. Yeh, A mathematical revision of the Landau-Lifshitz equation, Journal of Marine Science and Technology Vol. 17, no. 3, pp. 228-237, 2009.

\bibitem{ratiu} J.E. Marsden, T.S. Ratiu, {\it Introduction to Mechanics and Symmetry. A Basic Exposition of Classical Mechanical Systems}, 2d. edition, Springer-Verlag, New-York, 1998.

\bibitem{martinez} E. Martinez, L. Lopez-Diaz, L. Torres, C. Garcia-Cervera, Micromagnetic simulations with thermal noise. Physical and numerical aspects, J. Magn. Mag. Mat. 316 (2007), 269-272.

\bibitem{mercer} J.I. Mercer, M.L. Plumer, J.P. Whitehead, J. Van Ek, Atomic level micromagnetic model of recording media switching at elevated temperatures, Appl. Phys. Lett. 98 (2011).

\bibitem{milian} A. Milian, Stochastic viability and comparison theorem, Colloquium Mathematicum vol. LXVIII, Fasc. 2 (1995), 297--316.

\bibitem{mil} Milstein G.N., Repin YU.M., T. M. Symplectic integration of hamiltonian systems with additive noise. Society for
Industrial and Applied Mathematics, 39(6):2066-2088.

\bibitem{nelson} E. Nelson, {\it Dynamical theories of Brownian motion}, Princeton University Press, 1967.

\bibitem{oksendal2003stochastic} {\O}ksendal, B. 2003, {\it Stochastic Differential Equations: An Introduction with Applications}, Hochschultext / Universitext (Springer).

\bibitem{palis} J. Palis, W. De Melo, {\it Geometric theory of dynamical systems. An introduction}, Springer, 1982. 

\bibitem{stepanov} Y. Raikher, V. Stepanov, Magnetization dynamics of a single-domain particles by superparamagnetic theory, J. Magn. Mag. Mat. 316 (2007), 417-421.

\bibitem{raikher} Y. Raikher, V. Stepanov, R. Perzynskib, Dynamics hysteresis of a superparamagnetic nanoparticle, Phys. B 343 (2004), 262-266.

\bibitem{Saarinen} A. Saarinen, M-L. Linne, O. Yli-Harja, Stochastic Differential Equation Model for Cerebellar Granule Cell Excitability, Plos Computational Biology, Vol. 4, Issue 2, 1-11, 2008.

\bibitem{scholz} W. Scholz, T. Schrefl, J. Fidler, Micromagnetic simulation of thermally activated switching in fine particles, J. Magn. Mag. Mat. 233 (2001), 296-304.

\bibitem{sharma} S.N. Sharma et H. Parthasarathy, Dynamics of a stochastically perturbed two-body problem. Proceedings of the Royal Society A : Mathematical, Physical and Engineering Science, 463(2080), 979-1003, 2007.

\bibitem{smith} N. Smith, Modeling of thermal magnetization fluctuations in thin-film magnetic devices, J. Appl. Phys. 90 (2001), 5768-5773.

\bibitem{thygesen} U.Thygesen, A survey of Lyapunov techniques for stochastic differential equations,1997.

\bibitem{watanabe} N. Ikeda,  S. Watanabe,Stochastic Differential Equations and Diffusion Processes, North-Holland, 1989.  

\bibitem{Wiggins} C.Li, S.Wiggins, Invariant Manifolds and Fibrations for Perturbed Nonlinear Schr\"odinger Equations, Springer Science et Business Media, 1997.

\bibitem{wiggins2} S. Wiggins, {\it Introduction to Applied Nonlinear Dynamical Systems and Chaos}, Texts in Applied Mathematics Vol. 2, Springer, 2003.

\bibitem{zambrini} J-C. Zambrini, Stochastic Deformation of Classical Mechanics, Discrete and Continuous Dynamical Systems, 2013 supplement (2013), 807-813.

\bibitem{zheng} G.-P. Zheng, D. Gross, M. Li, Atomistic modeling of nanocrystalline ferromagnets, J. Appl. Phys. 93 (10), pp. 7652-7654, 2003.
\end{thebibliography}
\end{document}